\documentclass[final]{siamltex}
\usepackage{epsfig}
\usepackage{amsmath}
\usepackage{epic}
\usepackage{comment}
\usepackage{amssymb}
\newtheorem{condition}[theorem]{Condition}

\def\Hneg#1{\widetilde{H}^{-#1}}
\newcommand{\Abs}[1]{\left|#1\right|}
\newcommand{\abs}[1]{\lvert#1\rvert}

\newcommand{\innprd}[2]{\left< #1 , #2 \right>}
\newcommand{\binnprd}[2]{\left( #1 , #2 \right)}
\def\ge{\geqslant}
\def\le{\leqslant}

\def\norm#1{\left|\!\left| #1 \right|\!\right|}
\def\nnorm#1{|\!| #1 |\!|}

\def\op#1{{\mathcal #1}}
\def\vect#1{{\bf #1}}

\def\enorm#1{|\!|\!| #1 |\!|\!|}

\def\thref#1{\textnormal{(\ref{#1}\textnormal)}}

\title{Multigrid preconditioning of linear systems \\for interior point methods
applied to a class of box-constrained optimal control problems}

\author{Andrei Dr{\u{a}}g{\u{a}}nescu\thanks{Department 
    of Mathematics and Statistics, University of Maryland, Baltimore
    County, 1000~Hilltop Circle, Baltimore, Maryland 21250 ({\tt draga@umbc.edu}). The work of this author was supported 
    in part by the Department of Energy under contract no. DE-SC0005455, and by the National Science Foundation under awards
    DMS-1016177 and DMS-0821311.}
  \and{Cosmin Petra\thanks{Mathematics and Computer Science Division,
      Argonne National Laboratory, 9700 S Cass Avenue, Argonne, IL 60439 ({\tt petra@mcs.anl.gov}).
      The work of this author was supported in part by the National Science Foundation under award CCF-0728878.}}}

\begin{document}

\maketitle

\begin{abstract} In this article we construct and analyze multigrid 
preconditioners for discretizations of operators of the form
${\mathcal D}_{\lambda}+{\mathcal K}^*{\mathcal K}$, where $D_{\lambda}$ is the
multiplication with a relatively smooth function $\lambda>0$ and
${\mathcal K}$ is a compact linear  operator. These systems arise when
applying interior point methods  to the minimization problem
$\min_{u} \frac{1}{2}(|\!|{\mathcal K} u-f|\!|^2 +\beta|\!|u|\!|^2)$ with
box-constraints $\underline{u}\leqslant u\leqslant\overline{u}$ on the controls.
The presented preconditioning technique is  closely related to the one
developed by Dr{\u a}g{\u a}nescu and Dupont  in~\cite{MR2429872} for
the associated unconstrained problem, and is intended for large-scale
problems.  As in~\cite{MR2429872}, the quality of the resulting
preconditioners is shown to increase as $h\downarrow 0$,
%at a rate that is optimal with respect to $h$  if the meshes are uniform
but decreases as the smoothness of $\lambda$ declines.  We test this
algorithm first on a Tikhonov-regularized backward parabolic equation
with box-constraints on the control, and then on a standard elliptic-constrained optimization
problem.  In both cases it is shown that the number of linear
iterations per optimization step, as well as the total number of
fine-scale matrix-vector multiplications is decreasing with increasing
resolution, thus showing the method to be potentially very efficient
for truly large-scale problems.
\end{abstract}



\pagestyle{myheadings}
\thispagestyle{plain}
\markboth{A.~DR{\u A}G{\u A}NESCU AND C.~PETRA}{MULTIGRID PRECONDITIONING FOR 
IPMs}

%-----------SECTION-----------%
\section{Introduction} 
\label{mgimp:sec:intro}
%--------------------------%
In this work we present a multigrid preconditioning technique for solving linear
systems arising when applying interior point methods to the
control-constrained optimal control problem
\begin{eqnarray}
\label{mgipm:eq:maineq}
    \textnormal{minimize}\ \  \op{J}_{\beta}(u)\stackrel{\mathrm{def}}{=}
      \frac{1}{2}\nnorm{\op{K} u-f}^2 +  \frac{\beta}{2}\nnorm{u}^2,\ \ u\in \op{U}_{\mathrm{ad}}\ ,
\end{eqnarray}
where $\Omega\subset \mathbb{R}^d$ is a bounded domain, $\beta>0$, $\op{K}:L^2(\Omega)\rightarrow L^2(\Omega)$ is 
a linear, compact operator, and the set of admissible solutions is given by
$$ \op{U}_{\mathrm{ad}} = \{u\in L^2(\Omega)\::\: 
\underline{u} \le u \le \overline{u}\ \  a.e.\},
$$ with $\underline{u}, \overline{u}\in L^{2}(\Omega)$, and $\underline{u} \le \overline{u}$.
The following examples form the main motivation for the present work:\\
\indent \emph{Example} A: {\bf Box-constrained time-reversal for a parabolic equation.}
Consider the linear parabolic initial value problem
\begin{equation}
\label{eq:parabolic_gen1}
\left\{
\begin{array}{lll}
\partial_t y  + \op{A} y = 0 &,&\ \ \mbox{on}\; \Omega\times (0,
\infty)\ , \\ y(x,t)= 0&,&\ \ \mbox{on}\; \partial\Omega\times (0,
\infty)\ , \\ y(x, 0) =u(x)&,&\ \ \mbox{for}\; x\in\Omega\ , \\
\end{array}
\right .
\end{equation}
where $\op{A}$ is a linear elliptic operator, and denote the solution
map by $\op{S}(t) u\stackrel{\mathrm{def}}{=}y(\cdot,t)$. 
To formulate a control problem we define $\op{K}=\op{S}(T)$, where $T>0$
is a fixed ``end-time'' $T>0$; the resulting optimization problem is controlled
by the initial value $u$.
Note that if the box constraints in~\eqref{mgipm:eq:maineq} are left out, i.e.,
$\op{U}_{\mathrm{ad}}=L^2(\Omega)$,  then~\eqref{mgipm:eq:maineq} is
the Tikhonov-regularized formulation of the  ill-posed problem
\begin{equation}
\label{eq:illposedgen}
\op{K} u = f\ .
\end{equation}
\indent \emph{Example} B: {\bf Elliptic-constrained distributed optimal control
problem.}\\  In this example we let  $$\op{K}=\Delta^{-1}\ ,$$ where
$\Delta$ is the Laplace operator acting on $H_0^1(\Omega)$. In this
case  the problem~\eqref{mgipm:eq:maineq} is usually formulated  as
the PDE-constrained optimization problem (e.g., see Borzi and
Schulz~\cite{MR2505585})
\begin{eqnarray}
\label{mgipm:eq:lap}
  \begin{array}{cc}\vspace{7pt}
    \textnormal{minimize}& \frac{1}{2}\norm{y-f}^2 +
    \frac{\beta}{2}\norm{u}^2\\   \textnormal{subj. to\ \ } &\Delta
    y=u\ ,\ y\in H_0^1(\Omega),\ \underline{u} \le u \le \overline{u}\
    \  a.e.\  \mathrm{in}\ \Omega.
  \end{array}
\end{eqnarray}

Our primary motivation is rooted in  solving large-scale inverse
problems like the one in Example~A, which is a simplified version of 
the problem considered by Ak{\c c}elik et
al. in~\cite{Bart_Omar:SC05a}. There the question was to identify the
initial concentration $u=y(\cdot,0)$ of a contaminant released in the atmosphere
in a given geographic area (the Los Angeles Basin) given later-time measurements at various
fixed locations. The spatio-temporal evolution of the concentration $y$ of the contaminant is
assumed, like in Example A, to be  governed by an advection-diffusion equation
with known wind-velocities and contaminant diffusivity
that can be formulated as~\eqref{eq:parabolic_gen1}. In Example~A we consider the case
where the measurements are taken at all points in space but  at a
single moment in time $T>0$, therefore the data $f$ is the entire state at time $T$. 
The problem thus becomes to invert a compact operator which, as is 
well known, is not continuously invertible. As a result, a naive approach to
inverting $\op{K}$ is unstable, in that small perturbations in the ``measurements'' 
$f$ result in exponentially large errors in the solution. If the ``exact'' measurements 
$f$ are resulted from applying $\op{K}$ to a ``true'' initial value $u$, that is, $f=\op{K}u$, then various regularization
techniques~\cite{MR97k:65145} are employed so that  the computed solution
$u_{\delta}$ of  the $\delta$-perturbed problem $\op{K} u_{\delta}=f_{\delta}$ (where 
$\nnorm{f-f_{\delta}}\le \delta$) converges to the ``true'' solution $u$
as $\delta\downarrow 0$, the most commonly  used being the
Tikhonov-regularization. One issue not resolved by the Tikhonov regularization is that
the solution $u_{\delta}$ of the regularized problem may exhibit non-physical, or otherwise qualitatively incorrect behavior: 
for example, if the concentration needs to have values in $[0,1]$, it is well known that $u_{\delta}$ may exceed these
limits. In addition, if the true initial
contamination event is localized, then $u_{\delta}$ oscillates around zero, and is not localized.
However, if explicit constraints  $\underline{u}=0$, $\overline{u}=1$  are
set, then the solution is  physically relevant, and is also localized. In Figure~\ref{fig:numerical_sol2} we show the solution of the
inverse problem where the target solution (true initial concentration) is represented by two
localized sources of contamination, with one dominating the other in
size.  The solution of the unconstrained problem (with $\delta$ down to roundoff error and $\beta$ optimized)
does not show two
sources of  pollution, unlike the solution of the constrained problem,
from which one can clearly identify two separate components in the
support of the initial value.  In addition, in the presence of
nonlinearities, with reaction terms that are sensitive  to
the sign of their arguments,  the importance of physically meaningful
box constraints cannot  be overstated. A similar situation is
encountered in image-deblurring, where  the target (gray-scale)
solution takes values in $[0,1]$; if these boundaries are strictly
enforced, then the quality of inversion (deblurring) is significantly
improved~\cite{MR1273034}.
\begin{figure}[!htb]
\begin{center}
        \includegraphics[width=5.5in]{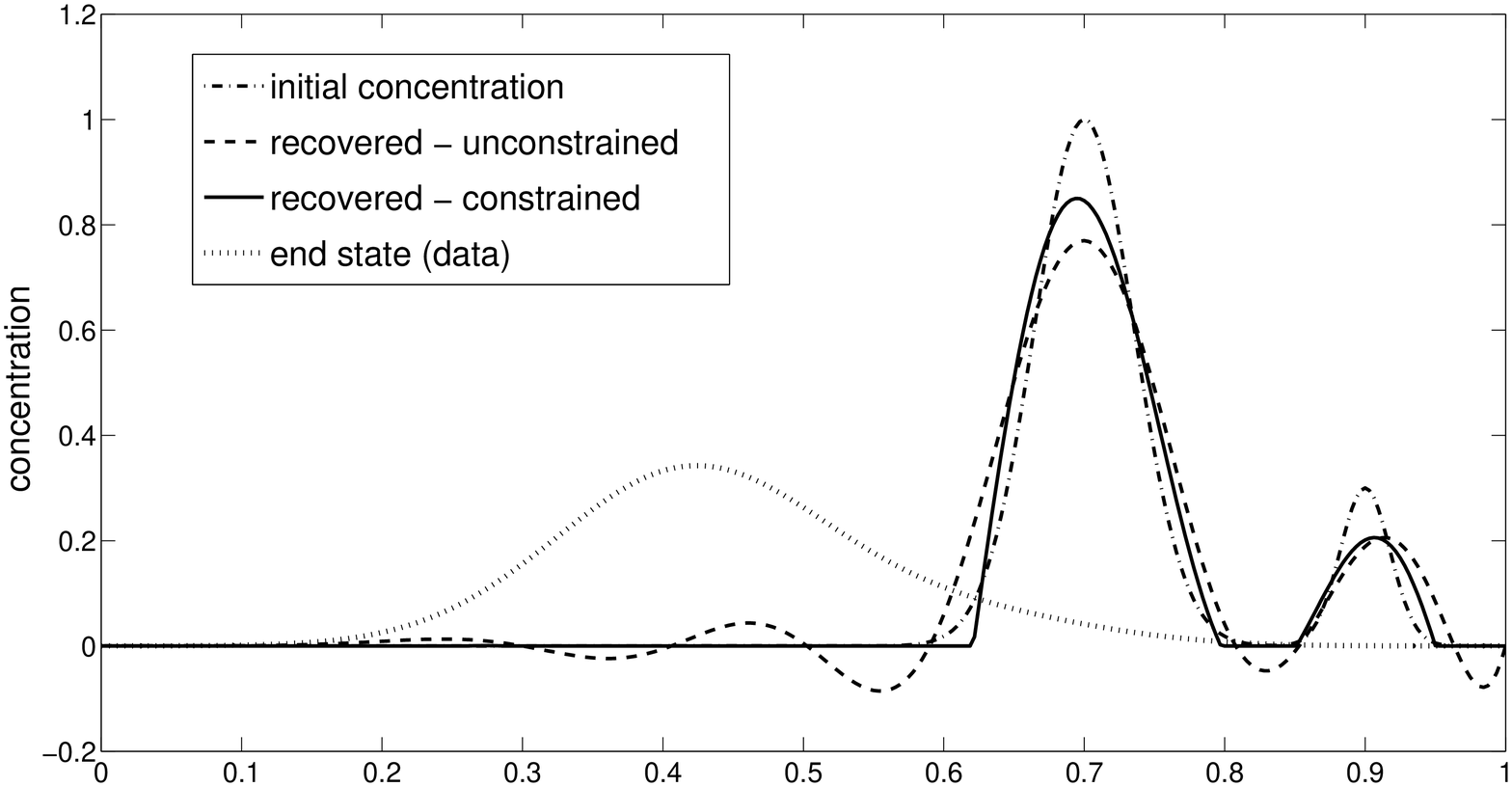}%constr_unconstr3.eps
\caption{Solution of the constrained vs. unconstrained inverse
problem, with the target solution {\textnormal(}initial concentration{\textnormal)} being a smooth function
supported on two disjoint intervals.}
\label{fig:numerical_sol2}
\end{center}
\end{figure}

The unconstrained optimization problem, where 
$\op{U}_{\mathrm{ad}}=L^2(\Omega)$ in~\eqref{mgipm:eq:maineq},
ultimately reduces to solving the normal equations
\begin{equation}
\label{eq:normaleq}
(\beta I+\op{K}^*\op{K})u=\op{K}^*f\ .
\end{equation}
For compact operators these systems either represent or resemble very
well integral equations of the second kind.  Starting with the works
of Hackbusch~\cite{hackbusch81} (see also~\cite{MR1350296})
much effort  has been devoted
to developing efficient multigrid methods for
solving~\eqref{eq:normaleq}, e.g., see~\cite{MR1151773, MR97k:65299,
MR2001h:65069, MR1986801, MR2421947, MR2429872} and the references
therein. We should point out that multigrid methods were originally developed
for solving elliptic equations~\cite{MR814495, MR2001h:65002}, and later significant
efforts were devoted to extending these methods to other important differential 
equations such as advection-diffusion and the Navier-Stokes equations (e.g., see~\cite{MR2155549} and
the references therein). For elliptic equations, typically  
the goal is to reduce the condition number of the discrete system from $O(h^{-2})$ 
to $O(1)$, which results in a solution process that solves the equation
in a number of iterations that is mesh-independent. However, the 
discrete version of~\eqref{eq:normaleq} has a condition number which is $O(\beta^{-1})$,
with the bound being independent of $h$. Moreover, even for $\beta=0$, conjugate gradient (used as a regularizer)
already solves~\eqref{eq:normaleq} in a mesh-independent number of 
iterations~\cite{MR97k:65145}; in other words, mesh-independence is nearly effortless for integral equations.
Instead, for systems like~\eqref{eq:normaleq} multigrid is used
to further reduce the condition number of the preconditioned system. For example, 
in~\cite{MR2429872} it is shown that by using specially designed multigrid preconditioners, the preconditioned version 
of~\eqref{eq:normaleq}
has a condition number bounded by $O(h^2/\beta)$, the consequence of which
 is interesting at least from a theoretical point of view: if
$\beta>0$ is kept fixed  (which is normally not the case in practical
applications)  the number of iterations required to solve the problem
{\bf decreases with $h\downarrow 0$} to the point where only one iteration would be enough to solve the problem
with sufficient accuracy.  This fact, already known to
Hackbusch~\cite{hackbusch81}, constitutes a departure from the usual
multigrid framework, where, as stated before, the goal is to achieve mesh-independence ({\bf bounded} number of iterations as $h\downarrow 0$). 
%However, this behavior of multigrid for integral equations should not
%be surprising, since unpreconditioned conjugate gradient already
%solves~\eqref{eq:normaleq} in a mesh-independent number of 
%iterations~\cite{MR97k:65145}. 
Hence, as a result of specific multigrid preconditioning, 
the solution process for the unconstrained problem~\eqref{eq:normaleq} requires fewer and fewer fine-scale \emph{matrix-vector multiplications}
(mat-vecs)
as $h\downarrow 0$. The main contribution of the
present work is to show  that such performance can also be achieved  in
the presence of explicit  box constraints on the control, as
formulated in~\eqref{mgipm:eq:maineq}.

Multigrid methods have long been associated with solving large-scale
problems, and  beginning with the work of Hackbusch~\cite{MR525024,
MR600205}, and especially over the last decade, significant efforts were
concentrated on devising efficient multilevel methods for optimal
control  problems~\cite{MR2038939, MR2160699, MR2123799, MR2395267,
MR2491821, benhabtar09}. A more detailed discussion of the subject and many references 
can be found in the recent
review article by Borzi and Schulz~\cite{MR2505585}.  While most --
not all -- of the aforementioned articles discuss unconstrained
problems, the addition of box constraints  pose additional
challenges associated with the presence of the
Lagrange multipliers related to the inequality-constraints,
which in general are less regular than the solutions.  Optimization methods for such
problems with bound constraints typically fall in one of two  categories:
active-set type strategies, especially \emph{semismooth Newton methods} (SSNMs) and \emph{interior point
methods} (IPMs).
Over the last decade both IPMs and SSNMs have consistently attracted the attention 
of the scientific community due to their proven efficiency in solving 
distributed optimal control problems with PDE constraints. Both strategies lead
to superlinear local convergence~\cite{MR1776662, MR2421314, MR1972219}
and lend themselves to analysis both in a finite dimensional setting and in function
space~\cite{MR1972217, MR2193506} (see also \cite{MR2516528}), which is a critical stepping stone
to proving mesh-independence~\cite{ MR2085262}. 
Each of IPMs and SSNMs consists of  an outer
iterative process that further requires solving one or two linear systems  at
each outer iteration. For large-scale problems the solution of these linear 
systems often becomes the bottleneck of the computation.
In terms of their linear algebra needs, IPMs and SSNMs
 exhibit significant differences and require
separate treatment. For SSNMs, the linear systems involve a subset of unknowns
and equations of the Hessian of the cost functional, while for IPMs 
the systems have the same structure as the systems arising in the 
unconstrained problem, but contain additional terms on the diagonal which
usually are a source of extreme ill-conditioning. The question of efficient multigrid
preconditioning of the linear systems arising in the semismooth Newton
solution process is the subject of current research~\cite{Dra:promise}.

The focus of the current work addresses  the question of
efficient multigrid preconditioning of linear systems arising in the IPM
solution process. 
Multigrid and IPMs
have been shown to work well together for elliptic variational
inequalities (obstacle problems)~\cite{MR2038939}, and for some
classes of problems where the Hessian of the cost functional is
elliptic~\cite{benhabtar09}.  Instead, the Hessian of the cost
functional in~\eqref{mgipm:eq:maineq} is a compact operator and
requires a significantly different approach.

In this article we treat  $\op{K}$ and its discretizations as
black-box operators, and we employ the {\bf first-discretize-then-optimize} strategy. 
We apply specific primal-dual IPMs to the discrete version 
of~\eqref{mgipm:eq:maineq}, and we develop  multigrid preconditioners
for the linear systems arising at each IPM iterate. 
As shown in Section~\ref{mgimp:sec:ipm}, if the discrete optimal control 
problem is formulated appropriately, then the linear systems to be solved 
involve matrices of the form $\vect{D}+\vect{K}^*\vect{K}$, where
$\vect{D}$ is a diagonal matrix, $\vect{K}$ is the discrete representation of 
$\op{K}$, and $\vect{K}^*$ is the adjoint of $\vect{K}$ with respect to a 
certain discrete inner product. When using standard finite elements, the matrices
$\vect{K}, \vect{K}^*$ are dense; consequently, for large-scale problems, 
they cannot be formed and the systems are solved using iterative 
methods. Since residuals can be computed at the equivalent cost of two applications of~$\vect{K}$, 
residual computation is expected to be very expensive,
therefore  efficient preconditioners are critical 
for minimizing the number of necessary mat-vecs.
Our strategy in this work is to adapt the multigrid techniques developed in~\cite{MR2429872}, 
where the matrix $\vect{D}$ had the form $\beta \vect{I}$. To analyse the resulting multigrid 
preconditioner we interpret $\vect{D}$ as being the discretization of an operator $\op{D}_{{\lambda}}$, where
$(\op{D}_{{\lambda}}u)(x) = \lambda(x) u(x)$ is the operator representing
pointwise multiplication of a function $u$ with  a smooth function\footnote{This can always be accomplished, $\lambda$ can be any smooth (here $C^2$ is sufficient)
interpolant of the discrete function representing the diagonal of $\vect{D}$.} $\lambda$. From a technical perspective, 
our main  accomplishment consists of showing that the
operators of the type $\op{K} \op{D}_{{\lambda}}$ together with
their discretizations satisfy a set of smoothing
conditions shown
in~\cite{MR2429872} to be sufficient for the multigrid preconditioner
to  have the desired qualities. 
%Therefore, our first goal is to discretize~\eqref{mgipm:eq:maineq} in such a way that the discrete
%systems maintain the above form, because these types of systems render themselves to good multigrid preconditioning. 

For simplicity and concreteness we restrict most of our study to the
two-dimensional case, and we consider a standard finite element
discretization for $\op{K}$ using triangular elements and continuous piecewise linear functions. As will result from the analysis, these techniques
can be easily generalized to three dimensions and rectangular
elements, however, the extension to higher degree finite elements is not obvious and
forms the subject of current research. 

This article is organized as follows: After formally introducing the
discrete optimization problem in Section~\ref{mgipm:sec:problem}, we
discuss the specific linear algebra requirements of the interior point methods in
Section~\ref{mgimp:sec:ipm}. Section~\ref{mgipm:sec:two_grid} is
central to this work as it presents the analysis of the two-grid
preconditioner, the main result being Theorem~\ref{mgipm:th:twogrid}.
In Section~\ref{mgipm:sec:multigrid} we develop a multigrid
preconditioner that preserves the qualities of the two-grid preconditioner.
Further, we apply the methods to Examples A and B in
Section~\ref{mgimp:sec:num_examples} and show some numerical results. 
%We end with some concluding remarks.

%-----------SECTION-----------%
\section{Notation and discrete problem formulation} 
\label{mgipm:sec:problem}
%--------------------------%
Throughout this paper we shall denote by  $W_p^m(\Omega), H^m(\Omega),
H_0^m(\Omega)$ (with $p\in[1,\infty], m\in \mathbb{N}$) the standard
Sobolev spaces, while \mbox{$\nnorm{\cdot}$} and
$\innprd{\cdot}{\cdot}$ are the $L^2$-norm and inner product,
respectively.  Let $\Hneg{m}(\Omega)$ be the dual (with respect to the
$L^2$-inner product) of $H^m(\Omega)\cap H^1_0(\Omega)$ for $m> 0$.
If $X$ is a Banach space then $\mathfrak{L}(X)$ denotes the space of bounded
linear operators on $X$. We regard square $n\times n$ matrices as operators in
$\mathfrak{L}(\mathbb{R}^n)$ and we write matrices and vectors using bold font. If $\vect{A}$ is a symmetric positive definite matrix, 
we denote by $\binnprd{\vect{u}}{\vect{v}}_{\vect{A}}=\vect{v}^T \vect{A} \vect{u}$ 
the $\vect{A}$-dot product of two vectors $\vect{u}, \vect{v}$, and by
$\abs{\vect{u}}_{\vect{A}} = \sqrt{\binnprd{\vect{u}}{\vect{u}}_{\vect{A}}}$ 
the $\vect{A}$-norm; if $\vect{A}=\vect{I}$ we drop the subscript from the inner product
and norm. The space of $m\times n$ matrices is denoted by $M_{m\times n}$; if $m=n$ we write
$M_{n}$ instead of $M_{m\times n}$. Given some norm $\nnorm{\cdot}_s$ on a vector
space $\op{X}$, and $T\in\mathfrak{L}(\op{X})$, we denote by $\nnorm{T}_s$ the induced operator-norm
$$\nnorm{T}_s = \sup_{u\in \op{X},\ \nnorm{u}_s = 1}\nnorm{T u}_s\ .$$
Consequently, if $T\in\mathfrak{L}(L^2(\Omega))$ then $\nnorm{T}$ (no subscripts) is the $L^2$ operator-norm
of~$T$. If $\op{X}$ is a  Hilbert space and $T\in \mathfrak{L}(\op{X})$ then
$T^*\in \mathfrak{L}(\op{X})$ denotes the adjoint of $T$.

We assume that  $\Omega\subset \mathbb{R}^2$ is a bounded, polygonal domain and that $\op{K}$
is discretized using continuous  piecewise linear functions on
triangular elements. We consider the usual multigrid framework where the operator
is discretized at several resolutions. Let~${\mathcal T}_{h_0}$ be a triangulation of the domain $\Omega$, and
define ${\mathcal T}_{h/2}$ inductively to be the Goursat
refinement of ${\mathcal T}_{h}$ for all $h\in I$ with
$$I= \{h_0/2^i: i=0, 1, 2,\dots\}\ ,$$ 
where each triangle in $T\in {\mathcal T}_{h}$ is cut along
the three lines obtained by joining the midpoints of its edges.  Note
that $({\mathcal T}_{h})_{h\in I}$ is a  quasi-uniform
triangulation and the usual approximations hold~\cite{MR2373954}. We define
\begin{eqnarray*}
%\label{def:vhlin}
{\mathcal V}_h& =& \{u\in {\mathcal C}(\overline{\Omega})\ :\  \forall
T\in {\mathcal T}_{h}\ \ u|_T\  \mathrm{is\ linear}, \ \mathrm{and}\ u|_{\partial \Omega} \equiv 0
\}\ ,
\end{eqnarray*}  
so that $\ {\mathcal V}_{h/2} \subset {\mathcal V}_{h} \subset H_0^1(\Omega)$. 
We should note that zero-boundary conditions for the controls are consistent with the examples considered, 
and present a convenient framework for the analysis, but alternate 
boundary conditions can be considered. For Example~A,  setting $u$ to be zero on $\partial\Omega$ is natural, while
in Example~B  the boundary values of the control do not enter the discrete problem unless higher order cubatures
are used in the discretization. 
%Additional, problem-dependent boundary conditions may further be imposed, although
%they do not play an active role for either of the two applications considered: Example~A calls
%for zero-boundary conditions, while in Example~B boundary values of the control do not enter
%the problem.
If $N_h=\mathrm{dim}(\op{V}_h)$ and $P^h_1, \dots, P^h_{N_h}$ are the nodes of $\op{T}_h$ that lie in the interior of 
$\Omega$ let $\op{I}_h:\op{C}(\Omega)\rightarrow \op{V}_h$  be the standard interpolation operator 
$$\op{I}_h(u)=\sum_{i=1}^{N_h} u(P^h_i)\varphi_i^h\ ,$$ 
where $\varphi_i^h, i=1,\dots, N_h$ are the standard nodal basis functions. 
Given a family of positive weight-functions $(w_h)_{h\in I}\subset\op{V}_h$ 
we define the mesh-dependent inner products
\begin{eqnarray*}
\label{def:meship}
\innprd{u}{v}_h = \sum_{i=1}^{N_h} w_h(P^h_i)\: u(P^h_i) v(P^h_i),\ \ \mathrm{for}\
u, v\in {\mathcal V}_h\ ,
\end{eqnarray*} 
and let $\enorm{u}_h = \sqrt{\innprd{u}{u}_h}$. 
In order to satisfy $\innprd{\cdot}{\cdot}_h \approx \innprd{\cdot}{\cdot}$ as close as possible we replace 
exact integration on each triangle $\Delta P_1 P_2 P_3$ with the cubature
$$
\int_T f(x) dx \approx Q(f)=\frac{\mathrm{area}(T)}{3} \sum_{i=1}^3 f(P_i)\ .
$$
This defines the 
weight functions $w_h$ 
\begin{equation}
\label{eq:weighdef}
w_h(P^h_i) = \frac{1}{3}\sum_{P_i^h\in T\in\op{T}_h} \mathrm{area}(T)\ .
\end{equation}
Since the cubature $Q$ is exact for linear functions~\cite{MR0327006} we have
\begin{eqnarray*}
%\label{eq:weighdef}
\innprd{u}{v}_h = \int_{\Omega} \op{I}_h (u v),\ \mathrm{for\ all\ }u, v\in\op{V}_h\ .
\end{eqnarray*}
Moreover, if the grids are  quasi-uniform, then $h^{-2}w_h$ are uniformly bounded and bounded away from
$0$ with respect to $h\in I$, therefore by Lemma 6.2.7 in~\cite{MR2373954} there exist positive constants $C_1, C_2$
independent of $h$ such that
\begin{equation}
\label{eq:equivmeshipl2}
C_1\nnorm{u}\le \enorm{u}_h\le C_2\nnorm{u},\ \forall u \in\op{V}_h\ .
\end{equation}
We should point out that the norm-equivalence~\eqref{eq:equivmeshipl2}
extends to operator norms for operators in $\mathfrak{L}(L^2(\Omega))$, which allows us to 
interchange $\nnorm{T}$ with $\enorm{T}_h$ when needed as long as we factor in a 
mesh-independent constant.

We assume that for each $h\in I$ is given  a natural discretization  $\op{K}_h\in
\mathfrak{L}(\op{V}_h)$ of $\op{K}$, so that $\op{K}, \op{K}_h$ satisfy 
the \emph{Smoothed Approximation Condition} (SAC):
%%---------------------------------------%
\begin{condition}[SAC]
\label{mgipm:cond:condsmooth} An operator $\op{M}$ together with its discretization $\op{M}_h$ is said to satisfy
the Smoothed Approximation Condition if there exists a constant $C(\op{M})$ depending on $\op{M}, 
\Omega, \op{T}_{h_0}$ and independent of $h$
 so that 
\begin{enumerate}
\item[{\bf [a]}] smoothing:
\begin{equation}
\label{mgipm:cond:par_smooth}
\nnorm{\op{M} u}_{H^m(\Omega)} \le C(\op{M}) \norm{u},\ \ \forall u\in L^2(\Omega),\  m=0, 1, 2\ ;
\end{equation}
\item[{\bf [b]}] smoothed approximation:
\begin{equation}
\label{mgipm:cond:consist}
\nnorm{\op{M} u - \op{M}_h u}_{H^m(\Omega)}  \le C(\op{M}) h^{2-m}\norm{u}
\ \ \forall u\in {\mathcal V}_h,\  m=0, 1,\ h\in I\ .
\end{equation}
\end{enumerate}
\end{condition}
%%---------------------------------------%
%
Given two discrete functions $\underline{u}_h, \overline{u}_h\in\op{V}_h$ 
representing $\underline{u}, \overline{u}$, respectively, we now define the discrete
optimization problem
\begin{eqnarray}
\label{eq:maineqdiscrete}
    \textnormal{minimize}\ \  \op{J}^h_{\beta}(u)\stackrel{\mathrm{def}}{=}
      \frac{1}{2}\enorm{\op{K}_h u-f_h}_h^2 +  \frac{\beta}{2}\enorm{u}_h^2,\ \ u\in \op{U}^h_{\mathrm{ad}}\ ,
\end{eqnarray}
where $f_h\in\op{V}_h$ represents $f$ and the set of discrete admissible solutions is given by
$$ \op{U}^h_{\mathrm{ad}} = \{u\in \op{V}_h\::\: 
\underline{u}_h(P_i^h) \le u(P_i^h) \le \overline{u}(P_i^h)\ \  \forall i=1,\dots, N_h\}\ .
$$
The formulation~\eqref{eq:maineqdiscrete} needs a few comments. First we remark that the use of the discrete
norm $\enorm{\cdot}_h$ instead of $\norm{\cdot}$ is essential in order for  the linear systems  to be solved 
at each outer iteration to have a form amenable to efficient multigrid preconditioning. In other words, we have chosen
a discretization that allows for an efficient solution process. 
%Second, the choice of weights is important not only for
%convergence of the discrete solution of~\eqref{eq:maineqdiscrete} to that of~\eqref{mgipm:eq:maineq}, but also for the efficiency of the multigrid algorithm. 
Second, if $\underline{u}$ and $-\overline{u}$ are convex and continuous (e.g., when they are constant), 
then  the choice 
\begin{equation}
\label{eq:discrboundsdef}
\underline{u}_h = \op{I}_h(\underline{u}),\  \ \ \overline{u}_h=\op{I}_h(\overline{u})
\end{equation}
implies
$\op{U}^h_{\mathrm{ad}}\subset \op{U}_{\mathrm{ad}}$. This construction can be easily generalized to three dimensions and/or
tensor-product finite elements. 

We obtain a matrix formulation of~\eqref{eq:maineqdiscrete} by
representing  all vectors and operators using the standard nodal basis
functions $\varphi_i^h, i=1,\dots, N_h$.  More precisely, if we define
$T:\mathbb{R}^{N_h}\rightarrow \op{V}_h$ by
$$T(\vect{u}) = \sum_{i=1}^{N_h} u_i \varphi_i^h,\ \ \mathrm{where}\ \vect{u}=[u_1,\dots, u_{N_h}]^T\ ,$$
then the matrix \mbox{$\vect{K}_h=T^{-1}\op{K}_h T$}, regarded as an operator in $\mathfrak{L}(\mathbb{R}^{N_h})$,
 represents $\op{K}_h$ with respect to the nodal basis. 
If $\vect{W}_h$ is the diagonal matrix with diagonal entries $(w_h(P_i^h))_{1\le i\le N_h}$, and 
$\vect{f}_h, \underline{\vect{u}}_h, \overline{\vect{u}}_h$ represent $f_h, \underline{u}_h, \overline{u}_h$ respectively, 
then~\eqref{eq:maineqdiscrete} is equivalent to 
\begin{eqnarray}
\label{eq:maineqdiscretevect}
    \textnormal{minimize}\ \  J_{\beta}(\vect{u})\stackrel{\mathrm{def}}{=}
      \frac{1}{2}\abs{\vect{K}_h\vect{u}- \vect{f}_h}_{\vect{W}_h}^2 +  \frac{\beta}{2}
\abs{\vect{u}}_{\vect{W}_h}^2,\ \ \underline{\vect{u}}_h\le \vect{u}\le\overline{\vect{u}}_h\ ,
\end{eqnarray}
where the inequality $\vect{u}\le \vect{v}$ between vectors is meant coordinate-wise. When operating 
on a single grid we will omit the subscript $h$ for matrices and vectors.

Existence and uniqueness of solutions for both~\eqref{mgipm:eq:maineq} and~\eqref{eq:maineqdiscrete}  follows from 
the fact that $\op{J}_{\beta},~\op{J}^h_{\beta}$ are uniformly convex and $\op{U}_{\mathrm{ad}},~\op{U}^h_{\mathrm{ad}}$
are convex sets (e.g., see Theorem~1.43 in~\cite{MR2516528}). Furthermore, cf. Lemma 1.12 in~\cite{MR2516528},
the solution $\widehat{u}$ of~\eqref{mgipm:eq:maineq} is characterized by the following condition: there exist
$\underline{\lambda}, \overline{\lambda}\in L^2(\Omega)$ so that
\begin{equation}
\label{eq:lambdacharsol}
\left\{
\begin{array}{l}\vspace{5pt}
(\beta I  + \op{K}^*\op{K}) \widehat{u} + \overline{\lambda} - \underline{\lambda} = \op{K}^* f\ ,\\\vspace{5pt}
\widehat{u}\ge \underline{u},\hspace{10pt} \underline{\lambda}\ge 0,\hspace{10pt}  \underline{\lambda}(\widehat{u}-\underline{u})=0\ \ a.e.,\\
\widehat{u}\le \overline{u},\hspace{10pt} \overline{\lambda}\ge 0,\hspace{10pt}  \overline{\lambda}(\overline{u}-\widehat{u})=0\ \  a.e.
\end{array}
\right .
\end{equation}

%\input{convergence}
%-----------SECTION-----------%
\section{Interior point methods and linear systems}
\label{mgimp:sec:ipm}
%--------------------------%
In this section we briefly discuss the specifics of the interior point method
we use for solving the discrete optimization problem~\eqref{eq:maineqdiscretevect},
and we describe in detail the linear systems that need to be solved at each outer
iteration. If we denote 
\begin{equation}
\label{eq:Adef}
\vect{A} \stackrel{\mathrm{def}}{=} \beta \vect{W}+\vect{K}^T  \vect{W}  \vect{K}\ ,
\end{equation}
then after a rearrangement of the terms in the objective function $J_{\beta}$
and dropping constant terms we write~\eqref{eq:maineqdiscretevect}
as a regular convex quadratic problem with affine constraints in  standard form:
\begin{eqnarray}
  \label{qp_form}
  \begin{array}{cl}\vspace{5pt}
    \textnormal{minimize}&  \frac{1}{2}\vect{u}^T\vect{A}\vect{u} - (\vect{K}^T  \vect{W}  \vect{f})^T \vect{u} \\
    \textnormal{subj to:\ \ } &\underline{\vect{u}}\le \vect{u} \le \overline{\vect{u}}\ .\\
  \end{array}
\end{eqnarray}
Since $\vect{W}$ is a diagonal matrix with positive entries, the matrix $\vect{A}$ is positive definite, therefore the above problem has a
unique solution (see \cite{NocedalWright:book}, p. 320).
Let us denote the Lagrangian corresponding to the QP \eqref{qp_form} by
\begin{eqnarray*}
  \label{eq:lagrangian}
    L(\vect{u},\vect{v}_1, \vect{v}_2) =
    \frac{1}{2} \vect{u}^T \vect{A} \vect{u} -
    (\vect{K}^T  \vect{W}  \vect{f})^T \vect{u} - \vect{v}_2^T (\overline{\vect{u}}-\vect{u})-\vect{v}_1^T ({\vect{u}}-\underline{\vect{u}}),
\end{eqnarray*} 
where $\vect{v}_1, \vect{v}_2$ are  vectors of non-negative multipliers corresponding to the inequality constraints.
Then the gradient and Hessian of the Lagrangian are given by
\begin{equation*}%\label{eq:lagrangianGrad}
\nabla L_\vect{u}(\vect{u}, \vect{v}_1, \vect{v}_2) =
\vect{A} \vect{u} - \vect{K}^T  \vect{W}  \vect{f} + \vect{v}_2- \vect{v}_1,\ \ \nabla^2 L_{\vect{u}\vect{u}}(\vect{u}, \vect{v})= \vect{A}\ .
\end{equation*}
Since the the linear independence constraint qualification holds, the unique solution~$\widehat{\vect{u}}$ of~\eqref{qp_form} satisfies
the Karush-Kuhn-Tucker (KKT) conditions
\begin{equation}\label{eq:KKT}
  \left\{
  \begin{array}{l}\vspace{5pt}
    \vect{A}\widehat{\vect{u}} + \widehat{\vect{v}}_2 - \widehat{\vect{v}}_1  = \vect{K}^T  \vect{W}  \vect{f} \\\vspace{5pt}
    \widehat{\vect{u}}\ge \underline{\vect{u}},\hspace{10pt} \widehat{\vect{v}}_1\ge 0,\hspace{10pt}  \widehat{\vect{v}}_1\cdot(\widehat{\vect{u}}-\underline{\vect{u}})=0\ ,\\
    \widehat{\vect{u}}\le \overline{\vect{u}},\hspace{10pt} \widehat{\vect{v}}_2\ge 0,\hspace{10pt}  \widehat{\vect{v}}_2\cdot(\overline{\vect{u}}-\widehat{\vect{u}})=0\ ,
  \end{array}
  \right .
\end{equation}
where $\widehat{\vect{v}}_1, \widehat{\vect{v}}_2$ are the 
multipliers, and $\vect{u}\cdot \vect{s}$ denotes the component-wise product.
Moreover, since $\nabla^2 L_{\vect{u}\vect{u}}$ is positive definite, the above KKT conditions are also sufficient.
The primal-dual IPM consists of solving the perturbed KKT system
\begin{equation}\label{eq:cp}
  \left\{
  \begin{array}{l}\vspace{5pt}
    \vect{A}{\vect{u}} + {\vect{v}}_2 - {\vect{v}}_1  = \vect{K}^T  \vect{W}  \vect{f}\ , \\\vspace{5pt}
    {\vect{u}} > \underline{\vect{u}},\hspace{10pt} {\vect{v}}_1 > 0,\hspace{10pt}
    {\vect{v}}_1\cdot({\vect{u}}-\underline{\vect{u}})=\mu \vect{e}\ ,\\
    {\vect{u}} < \overline{\vect{u}},\hspace{10pt} {\vect{v}}_2 > 0,\hspace{10pt}
    {\vect{v}}_2\cdot(\overline{\vect{u}}-{\vect{u}})=\mu \vect{e}\ ,
  \end{array}
  \right .
\end{equation}
whose one-parameter family of  solutions $(\widehat{\vect{u}}(\mu),\widehat{\vect{v}}_1(\mu), \widehat{\vect{v}}_2(\mu))$ defines the \emph{central path}. As usual,
$\vect{e}=[1,1,\dots,1]^T\in \mathbb{R}^{N_h}$. 
Practical IPM algorithms produce solutions that lie sufficiently close to the central path and converge rapidly to 
$(\widehat{\vect{u}},\widehat{\vect{v}}_1, \widehat{\vect{v}}_2)$.
An example of such method is Mehrotra's  predictor-corrector algorithm. Initially introduced for linear
programming~\cite{mehrotra92}, the method was successfully adapted to convex QPs and emerged in the last fifteen years as the (arguably) most practical and 
 efficient algorithm for this class of problems. For this project we used  Matlab to implement Mehrotra's method for
convex QPs from OOQP (see \cite{GerWri03:ooqp} for details).

To describe the method we first consider the the linear system defining the Newton direction
$(\delta\vect{u}, \delta\vect{v}_1, \delta\vect{v}_2)$ for~\eqref{eq:cp}:
\begin{equation}
\label{eq:lin_sys_aug_pd}
  \left \{
  \begin{array}{rcl}\vspace{5pt}
    \vect{A}\delta\vect{u} + \delta\vect{v}_2- \delta\vect{v}_1  &=&  \vect{K}^T  \vect{W}  \vect{f}
    -\vect{A} \vect{u} -\vect{v}_2 + \vect{v}_1\ ,  \\\vspace{5pt}
    \vect{V}_1 \delta\vect{u} + (\vect{U}-\underline{\vect{U}}) \delta\vect{v}_1 & = &\mu\vect{e}-\vect{v}_1\cdot(\vect{u}-\underline{\vect{u}}) ,\\
    -\vect{V}_2 \delta\vect{u} + (\overline{\vect{U}}-{\vect{U}}) \delta\vect{v}_2 & = &\mu\vect{e}-\vect{v}_2\cdot(\overline{\vect{u}}-{\vect{u}}) \ ,
  \end{array}
  \right .
\end{equation}
where $\vect{U}, \underline{\vect{U}}, \overline{\vect{U}}, \vect{V}_1$, and $\vect{V}_2$ are diagonal matrices with the diagonal given 
by the vectors $\vect{u}, \underline{\vect{u}}, \overline{\vect{u}}, \vect{v}_1$, and $\vect{v}_2$, respectively.
%The generic primal-dual approach presented above  requires a feasible starting point, that is, a point $\vect{u}, \vect{v}_1, \vect{v}_2$
%that satisfies the top equation and the inequalities in~\eqref{eq:cp}. Even though such points can be obtained by artificially transforming
%the problem, certain "infeasible" interior-point methods require a starting point that satisfies only the inequalities in~\eqref{eq:cp}.
In Mehrotra's algorithm, given  the  current iterate $(\vect{u},\vect{v}_1, \vect{v}_2)$, one first computes  the predictor direction 
$(\delta\vect{u}^{a},\delta\vect{v}_1^{a},\delta\vect{v}_2^{a})$  as the solution of~\eqref{eq:lin_sys_aug_pd} with $\mu=0$.
Secondly, the corrector direction $(\delta\vect{u},\delta\vect{v}_1,\delta\vect{v}_2)$ is the solution of a linear system that differs
from~\eqref{eq:lin_sys_aug_pd} only in the right-hand side, namely:
\begin{equation}
\label{lin_sys_aug_corr}
  \left \{
  \begin{array}{rcl}\vspace{5pt}
    \vect{A}\delta\vect{u} + \delta\vect{v}_2- \delta\vect{v}_1  &=&  \vect{K}^T  \vect{W}  \vect{f}
    -\vect{A} \vect{u} -\vect{v}_2 + \vect{v}_1\ ,  \\\vspace{5pt}
    \vect{V}_1 \delta\vect{u} + (\vect{U}-\underline{\vect{U}}) \delta\vect{v}_1 & = &\sigma \mu\vect{e}-\vect{v}_1\cdot(\vect{u}-\underline{\vect{u}})
    +\delta\vect{u}^{a}\cdot\delta\vect{v}_1^{a},\\
    -\vect{V}_2 \delta\vect{u} + (\overline{\vect{U}}-{\vect{U}}) \delta\vect{v}_2 & = &\sigma \mu\vect{e}-\vect{v}_2\cdot(\overline{\vect{u}}-{\vect{u}})
    +\delta\vect{u}^{a}\cdot\delta\vect{v}_2^{a}\ ,
  \end{array}
  \right .
\end{equation}
%\begin{equation}
%\label{lin_sys_aug_corr1}
%    \begin{array}{rcl}
%    (\beta \vect{W}+\vect{K}^T  \vect{W}  \vect{K})\delta\vect{u} - \delta\vect{v}  &=&
%      -((\beta \vect{W}+\vect{K}^T  \vect{W}  \vect{K})\vect{u} + \vect{K}^T  \vect{W}  \vect{f} - \vect{v}) \\
%      \vect{V} \delta\vect{u} + \vect{U} \vect{v} & = &\sigma\mu\vect{e}-\vect{u}\cdot\vect{v}+ \delta\vect{u}^{a}\cdot\delta\vect{v}^{a},
%    \end{array}
%\end{equation}
%\begin{displaymath}
%where $\mu = \frac{(\vect{u}-\underline{\vect{u}})^T\vect{v}_1 + (\overline{\vect{u}}-\vect{u})^T\vect{v}_2}{2n}$
where $\mu = \left((\vect{u}-\underline{\vect{u}})^T\vect{v}_1 + (\overline{\vect{u}}-\vect{u})^T\vect{v}_2\right)/(2 N_h)$,
%\end{displaymath}
%$\mu = \vect{u}^T\vect{v}/n$ \emph{(Cosmin, what is $\mu$ if we have two $v_1, v_2$ instead of just $v$ ?)}
and $\sigma >0$ is a centering parameter that is computed accordingly to Mehrotra's heuristic.
Therefore, both the predictor and the corrector step involve a system -- the  \emph{augmented system} -- of the form
\begin{equation}
\label{eq:linsys}
\left[
  \begin{array}{ccc}\vspace{5pt}
    \vect{A}& -\vect{I}& \vect{I}\\\vspace{5pt}
    \vect{V}_1& (\vect{U}-\underline{\vect{U}})& \vect{0}\\
    -\vect{V}_2 & \vect{0}& (\overline{\vect{U}}-{\vect{U}})
  \end{array}
  \right]
\cdot
\left[
  \begin{array}{l}\vspace{5pt}
    \delta\vect{u}  \\\vspace{5pt}
    \delta\vect{v}_1 \\
    \delta\vect{v}_2
  \end{array}
  \right] =
\left[
  \begin{array}{l}\vspace{5pt}
    \vect{r}_u  \\\vspace{5pt}
    \vect{r}_{v_1} \\
    \vect{r}_{v_2}
  \end{array}
  \right]\ .
\end{equation}
Since the matrices on the second and third block-rows of~\eqref{eq:linsys} are diagonal, we can substitute
$\delta\vect{v}_1 = (\vect{U}-\underline{\vect{U}})^{-1}(\vect{r}_{v_1}-\vect{V}_1\delta \vect{u})$
and $ \delta\vect{v}_2 = (\overline{\vect{U}}-\vect{U})^{-1}(\vect{r}_{v_2}+\vect{V}_2\delta \vect{u})$
into the first block-row to obtain the \emph{reduced system}
\begin{equation}
\label{lin_sys_normal_eqn1}
\left[\vect{A}+(\vect{U}-\underline{\vect{U}})^{-1} \vect{V}_1 + (\overline{\vect{U}}-\vect{U})^{-1}\vect{V}_2\right]\delta \vect{u} = \vect{r}
\end{equation}
with
$$
\vect{r}=\vect{r}_u + (\vect{U}-\underline{\vect{U}})^{-1} \vect{r}_{v_1} - (\overline{\vect{U}}-\vect{U})^{-1}\vect{r}_2\ .
$$

We note that the matrix of the reduced system~\eqref{lin_sys_normal_eqn1} is symmetric positive definite, while the
matrix of~\eqref{eq:linsys} is similar to the symmetric indefinite matrix
%\begin{equation}
%\label{eq:linsysindef}
$$
\vect{C}=\left[
  \begin{array}{ccc}\vspace{5pt}
    \vect{A}& \vect{I}& \vect{I}\\\vspace{5pt} \vect{I}&
    -\vect{V}_1^{-1}(\vect{U}-\underline{\vect{U}})& \vect{0}\\
    \vect{I}& \vect{0}&
    -\vect{V}_2^{-1}(\overline{\vect{U}}-{\vect{U}})
  \end{array}
  \right]\ .
$$
If strict complementarity holds for at least one coordinate in the pair
($\widehat{\vect{u}}, \widehat{\vect{v}}_1$)
(resp. ($\widehat{\vect{u}}, \widehat{\vect{v}}_2$)), then the diagonal
matrix $\vect{V}_1^{-1}(\vect{U}-\underline{\vect{U}})$ (resp. $\vect{V}_2^{-1}(\overline{\vect{U}}-\vect{U})$)
has increasingly small and/or large
entries as the interior-point algorithm approaches the solution.
It is easy to see that the largest eigenvalue of $\vect{C}$ is larger than any
of the diagonal entries of $\vect{V}_1^{-1}(\vect{U}-\underline{\vect{U}})$ or $\vect{V}_2^{-1}(\overline{\vect{U}}-\vect{U})$,
while the smallest (in absolute value) of its eigenvalues are $O(1)$. Therefore, the matrix $\vect{C}$, and hence the 
system~\eqref{eq:linsys} is ill-conditioned. Several
equivalent reformulations of the system~\eqref{eq:linsys} are
proposed in~\cite{benhabtar09} (also see~\cite{MR1361601}). 
As in~\cite{Bart_Omar:SC05a, MR2429872}, in this article we use 
the reduced form~\eqref{lin_sys_normal_eqn1} which, for the problem under study, is symmetric positive definite.
While~\eqref{lin_sys_normal_eqn1} also suffers from the well known ill-conditioning 
of interior point methods, and in fact is even more ill-conditioned than~\eqref{eq:linsys} if $\beta\ll 1$, the condition number 
of~\eqref{lin_sys_normal_eqn1} can be significantly reduced
first by rescaling both the unknowns and the equations  (left- and right-preconditioning), 
as shown below, and then by two-grid and multigrid preconditioning,
as discussed in Sections~\ref{mgipm:sec:two_grid} and~\ref{mgipm:sec:multigrid} .

Given a  vector $\vect{m}=[m_1,m_2,\dots]^T$ we denote by $\vect{D}_{\vect{m}}$ the diagonal matrix with entries
\mbox{$D_{ii}=m_i$}, and by $\vect{p}\cdot/\vect{m}$ the vector $[p_1/m_1,p_2/m_2,\dots]^T$. With $\vect{A}$ as in~\eqref{eq:Adef},
$\vect{w}=\vect{w}_h=[w_1,w_2,\dots, w_{N_h}]^T$, and
$$\vect{m} = \vect{v}_1\cdot/(\vect{u}-\underline{\vect{u}})  + \vect{v}_2\cdot /(\overline{\vect{u}}-\vect{u})\ ,$$
the reduced system~\eqref{lin_sys_normal_eqn1} can be written as
\begin{equation}
\label{eq:shortlinsysD}
  \left(\vect{D}_{\vect{m}+\beta \vect{w}}+\vect{K}^T \vect{W} \vect{K} \right) \delta\vect{u} =  \vect{r}\ .
\end{equation}
Left-multiplication with $\vect{W}^{-1}$ further yields
\begin{equation}
  \left(\vect{D}_{(\vect{m}/\vect{w})+\beta\vect{e}}+\vect{W}^{-1}\vect{K}^T \vect{W} \vect{K} \right)\delta\vect{u} =  \vect{W}^{-1}\vect{r}\ .
  \label{lin_sys_normal_genW}
\end{equation}
Let $\vect{p} = \sqrt{(\vect{m}/\vect{w})+\beta\vect{e}}$ (component-wise). By rescaling $\delta\vect{u}' = \vect{D}_{\vect{p}} \delta \vect{u}\ ,$
and factoring out $\vect{D}_{\vect{p}}$ in~\eqref{lin_sys_normal_genW}, the system becomes
\begin{equation}
  \left(\vect{I} +\vect{W}^{-1}\vect{L}^T \vect{W} \vect{L} \right)\delta\vect{u}' =
  \vect{D}_{1\cdot/\vect{p}} \vect{W}^{-1}\vect{r}\ ,
  \label{lin_sys_normal_genWD}
\end{equation}
where $$\vect{L}=\vect{L}_h  = \vect{K}_h\:\vect{D}_{1\cdot/\vect{p}}\ ,$$
and we used the commutation of the diagonal matrices $\vect{W}^{-1}$ and $\vect{D}_{1\cdot/\vect{p}}$.
We prefer to write~\eqref{lin_sys_normal_genWD} in  compact form as
\begin{equation}
  \left(\vect{I} +\vect{H}\right)\delta\vect{u}' =  \vect{r}'\ ,
  \label{lin_sys_normal_genWDH}
\end{equation}
with
$$\vect{H} = \vect{H}_h \stackrel{\mathrm{def}}{=} (\vect{W}_h)^{-1}\vect{L}_h^T \vect{W}_h \vect{L}_h\ ,$$
and $\vect{r}' = \vect{D}_{1\cdot/\vect{p}} \vect{W}^{-1}\vect{r}$.
Note that the matrix in~\eqref{eq:shortlinsysD} is symmetric positive definite, while the matrix $\vect{H}$
is symmetric (and positive definite) with respect to the $\vect{W}$-dot product. 
Furthermore, it is easy to see that $\nnorm{\vect{H}}= O(\beta^{-1}\nnorm{\vect{K}}^2)$
which implies that
$$\mathrm{cond}(\vect{I}+\vect{H})=O(\beta^{-1}\nnorm{\vect{K}}^2)\ ,$$
independently of the mesh parameter $h$.
However, for the model problems considered,
the matrix $\left(\vect{I} +\vect{H}\right)$ is dense (in standard representation),
and therefore for large-scale problems it cannot be formed and/or stored. Matrix-vector multiplication can be performed at a cost equivalent to
two applications of the matrix $\vect{K}$, hence residual computations are expensive. So~\eqref{lin_sys_normal_genWDH}
has to be solved using iterative methods, and for increased efficiency  we need high-quality, matrix-free  preconditioners.
As it turns out, it is the system~\eqref{lin_sys_normal_genWDH} rather than~\eqref{eq:shortlinsysD} that renders itself to good multigrid preconditioning.
In the next sections
we develop  a multigrid preconditioner for~\eqref{lin_sys_normal_genWDH} under the assumption
that~$\vect{m}$ represents a positive and relatively ``smooth'' function~$\mu_h$.

\section{The two-grid preconditioner} 
\label{mgipm:sec:two_grid}
In this section we develop and analyze a two-grid preconditioner for the linear 
system~\eqref{lin_sys_normal_genWDH}. The work relies on the multigrid techniques
developed by Dr{\u a}g{\u a}nescu and Dupont in~\cite{MR2429872} for~\eqref{eq:normaleq}. As will be shown, the constructed
preconditioner has, under certain hypotheses, optimal order quality
with respect to the discretization parameter $h$.

%%%%%%%%%%%%%%%%%%%%%%%%%%%%%%%%%%%%%%%%%%%%%%%%%%%%%%%%%%%%
% SUBSECTION
%%%%%%%%%%%%%%%%%%%%%%%%%%%%%%%%%%%%%%%%%%%%%%%%%%%%%%%%%%%%
%--------------------------%
\subsection{Algorithm design}
%---------------------%
\label{ssec:tgheuristics}
For the purpose of algorithm design and analysis it is advantageous 
to regard~\eqref{lin_sys_normal_genWDH} as an equation in  $\op{V}_h$ rather than $\mathbb{R}^{N_h}$,
so we have to identify  the operator in $\mathfrak{L}(\op{V}_h)$ that is represented by $\vect{H}_h$.
First we define for $\lambda\in L^{\infty}(\Omega)$ the multiplication-by-$\lambda$ operator
$\op{D}_{\lambda}:L^2\rightarrow L^2$ by
$$
\op{D}_{\lambda} u = \lambda\: u\ ,
$$
and its discrete version $\op{D}^h_{\lambda}\in\mathfrak{L}(\op{V}_h)$ by
$$\op{D}^h_{\lambda} u = \op{I}_h \op{D}_{\lambda} u\ .$$
Given a vector $\vect{m}\in\mathbb{R}^{N_h}$ we define a function $\mu_h\in\op{V}_h$ by setting
$\mu_h(P^h_i) = \vect{m}_i$, $i=1,\dots,N_h$; it follows that the diagonal matrix $\vect{D}_{(\vect{m}/\vect{w})+\beta}$ represents the 
operator 
$$\op{D}^h_{(\mu_h/w_h)+\beta} = \op{I}_h \op{D}_{(\mu_h/w_h)+\beta}\ .$$ 
To simplify notation let
\begin{equation}
\label{eq:lambda_def}
\lambda_h = (\mu_h/w_h) +\beta\ .
\end{equation}
Then $\vect{L}_h=\vect{K}_h \vect{D}_{1\cdot/\vect{p}}$ represents the operator
\begin{equation}
\label{eq:Lhdef}
\op{L}_h = \op{K}_h \op{D}^h_{1/\sqrt{\lambda_h}}\ .
\end{equation}
If we denote by $\op{L}_h^*$ the dual of $\op{L}_h$ with respect to the 
$\innprd{\cdot}{\cdot}_h$-inner product, that is, 
$$
\innprd{\op{L}_h u}{v}_h = \innprd{u}{\op{L}^*_h v}_h\ ,\ \ \forall u, v\in \op{V}_h\ ,
$$
then $\op{L}_h^*$ is represented by $\vect{W}_h^{-1} \vect{L}_h^T \vect{W}_h$, 
so  $\vect{H}_h=(\vect{W}_h)^{-1} \vect{L}_h^T \vect{W}_h \vect{L}_h$ represents the  operator
\begin{equation}
\label{eq:Hhdef}
\op{H}_h \stackrel{\mathrm{def}}{=} \op{L}_h^*\op{L}_h\ .
\end{equation}
Hence, the operator we need to invert for solving~\eqref{lin_sys_normal_genWDH} is
\begin{equation}
\label{mgipm:eq:fineopdef}
\op{G}_h={I}+ \op{H}_h \ .
\end{equation}
Note that the operator $\op{G}_h$ is symmetric with respect to 
$\innprd{\cdot}{\cdot}_h$, i.e., $\op{G}_h = \op{G}_h^*$.

The idea behind the proposed two-grid preconditioner for $\op{G}_h$ lies in the ``smoothing''
properties of $\op{L}_h$. More precisely, we regard $\op{L}_h$ as a discretization of $$\op{L}=\op{K}\: \op{D}_{1/\sqrt{\lambda}}$$
for some function $\lambda$ for which $$\op{I}_h(\lambda) = \lambda_h\ .$$ 
If $\lambda$ (assumed to be $\ge \beta >0$) is relatively smooth (e.g., it can always be chosen to be~$C^2$),
then the multiplication operator $\op{D}_{1/\sqrt{\lambda}}$ is neither  smoothing nor roughening,
so the follow-on application of $\op{K}$ results in smoothing. An alternative point of view is that
$\op{D}_{1/\sqrt{\lambda}}$ is bounded in $\mathfrak{L}(L^2)$, and if $\op{K}$ is compact, then $\op{K} \op{D}_{1/\sqrt{\lambda}}$ is also compact.
Hence, it is natural to assume that $\op{H}_h=\op{L}_h^* \op{L}_h$ is ``smoothing'', even though $\op{L}_h^*$ has no
direct connection with the dual of $\op{K} \op{D}_{1/\sqrt{\lambda}}$ in $\mathfrak{L}(L^2)$.

We consider the $L^2$-orthogonal splitting of the discrete space
\begin{equation}
\label{eq:vhsplit}
\op{V}_h = \op{V}_{2h}\oplus \op{W}_{2h}\ ,
\end{equation}
and let $\pi=\pi_{2h}$ be the $L^2$-projector onto $\op{V}_{2h}$. Following~\cite{MR97k:65299, MR2001h:65069, mythesis, MR2429872}, 
we propose 
\begin{equation}
\label{mgipm:eq:twogridprec}
\op{N}_h=\rho+ \op{G}_{2h}\pi \ 
\end{equation}
as a two-grid preconditioner, where $\rho=\rho_{2h} = I-\pi_{2h}$ is the projector on $\op{W}_{2h}$, and
the coarse function $\lambda_{2h}$ entering the definition of $\op{G}_{2h}$ 
%through $\op{L}_{2h}= \op{K}_{2h} \op{D}^{2h}_{1/\sqrt{\lambda_{2h}}}$
is given by 
\begin{equation}
\label{eq:lambdacoarse}
\lambda_{2h} = \op{I}_h \lambda_h\ .
\end{equation}
The operator $\op{N}_h$ can also be regarded as an additive Schwartz preconditioner with respect to the splitting~\eqref{eq:vhsplit} (see~\cite{MR2001h:65069}). 
Moreover, the inverse of $\op{N}_h$ is given by
\begin{equation}
\label{mgipm:eq:twogridprecinv}
\op{S}_h\stackrel{\mathrm{def}}{=} \op{N}_h^{-1} = \rho+ \op{G}_{2h}^{-1}\pi \ .
\end{equation}
For developing a multigrid algorithm of comparable quality with the two-grid preconditioner we follow the same strategy as in~\cite{mythesis, MR2429872}, which we
briefly outline in Section~\ref{mgipm:sec:multigrid}.

As shown in the analysis, the use of the $L^2$-projector $\pi_{2h}$ in the definition of $\op{S}_h$, as opposed to other projectors or restriction operators
turns out to be critical for the quality of the preconditioner.
Unfortunately $\pi_{2h}$ is not symmetric with respect to $\innprd{\cdot}{\cdot}_h$, therefore $\op{N}_h$ and $\op{S}_h$ are 
symmetric neither with respect to $\innprd{\cdot}{\cdot}_h$ nor to $\innprd{\cdot}{\cdot}_{L^2}$, but they are almost symmetric.
At the root of this problem lies the use of mesh-dependent norms in the formulation of the discrete optimization problem~\eqref{eq:maineqdiscrete}, 
which in turn was necessary for  the  linear systems inner to the interior point algorithms to be of the form~\eqref{eq:shortlinsysD}, that is, to 
have a diagonal matrix $\vect{D}_{\vect{m}+\beta \vect{w}}$ 
added to $\vect{K}^T \vect{W}\vect{K}$. Had we used the exact $L^2$-norm in the discrete formulation~\eqref{eq:maineqdiscrete},
then the matrix in the system~\eqref{eq:shortlinsysD} would have had the form 
$\vect{D}_{\vect{m}}+\beta \vect{M}+\vect{K}^T \vect{M} \vect{K}$, where $\vect{M}$  is the mass matrix, 
and this form  is less convenient for
preconditioning.

In order to describe the matrices representing  $\op{N}_h$ and $\op{S}_h$
we consider the prolongation operator
$\vect{J}_{h}\in M_{N_h\times N_{2h}}$ representing the natural embedding of $\op{V}_{2h}$ into
$\op{V}_h$, and we define the restriction $\vect{R}_{2h}\in M_{N_{2h}\times N_h}$ 
by $\vect{R}_{2h}=2^{-d}\vect{J}_{h}^T$, $d$ being the dimension of the ambient space (here $d=2$).
Then $\pi_{2h}$ is represented  by the matrix
$$ \vect{\Pi}_{2h} = \vect{M}_{2h}^{-1}\cdot \vect{R}_{2h}\cdot \vect{M}_h\ ,$$
where $\vect{M}_{h}$ (resp. $\vect{M}_{2h}$) is the rescaled mass
matrix on the  fine (resp. coarse)  mesh, defined by
$(\vect{M}_{h})_{i j} = h^{-d}\innprd{\varphi^h_i}{\varphi^h_j}$. Note
that $\vect{\Pi}_{2h}$ is a matrix of type $N_{2h}\times N_h$, and that
$\vect{M}_{2h}=\vect{R}_{2h}\cdot \vect{M}_h\cdot \vect{J}_h$. Furthermore,
the { square} ${N_h\times N_h}$ projection matrix is  given by 
$\vect{P}_{h}= \vect{J}_h\cdot \vect{\Pi}_{2h}$, so that $\vect{P}_h^2 = \vect{P}_h$.  The
projector $\rho=(I-\pi)$ is represented by the matrix
$\vect{Q}_h=(\vect{I}-\vect{P}_h)\in M_{N_h\times N_h}$.  Finally, $\op{S}_h$
is represented by
\begin{equation}
\label{eq:matSdef}
\vect{S}_h = \vect{Q}_h + \vect{J}_h \vect{G}_{2h}^{-1}\vect{\Pi}_{2h}\ .
\end{equation}
Oftentimes in practice the exact projection $\vect{\Pi}_{2h}$ in~\eqref{eq:matSdef}  is replaced by the 
restriction~$\vect{R}_{2h}$, and $\vect{Q}_h$ is taken to be ($\vect{I}-\vect{R}_{2h}$). While for certain problems
this is a viable option~\cite{Bart_Omar:SC05a}, for the applications considered in this work the quality of the preconditioner 
is significanly diminished by this change.

%%%%%%%%%%%%%%%%%%%%%%%%%%%%%%%%%%%%%%%%%%%%%%%%%%%%%%%%%%%%
% SUBSECTION
%%%%%%%%%%%%%%%%%%%%%%%%%%%%%%%%%%%%%%%%%%%%%%%%%%%%%%%%%%%%
%--------------------------%
\subsection{Algorithm analysis}
%---------------------%
\label{ssec:tganalysis}
Our analysis consists of evaluating the quality of the two-grid preconditioner by eventually estimating the
spectral distance $d_{\sigma}(\op{S}_h, \op{G}_h^{-1})$ (Theorem~\ref{mgipm:th:twogrid}), and is performed over 
three steps. First we evaluate the norm-distance $\nnorm{\op{G}_h-\op{N}_h}$ under the assumption that $\op{H}_h$ satisfies 
Condition~\ref{mgipm:cond:cond1} below. Second, we show that if $\op{L}$ and $\op{L}_h$ verify Condition~\ref{mgipm:cond:condsmooth}, then 
$\op{H}_h$ satisfies 
Condition~\ref{mgipm:cond:cond1}. Third, we show that $\op{L}, \op{L}_h$ satisfy Condition~\ref{mgipm:cond:condsmooth} (SAC) with a 
constant $C(\op{L})$ depending on $\lambda$ and on $C(\op{K})$, that is, the constant associated  to $\op{K}, \op{K}_h$ satisfying
SAC. For specific applications, the fact that $\op{K}, \op{K}_h$ satifsy SAC is normally verifiable, as is shown
in Section~\ref{mgimp:sec:num_examples}.
%Throughout this section we assume that the operators $\op{H}_h$ satisfy
%%---------------------------------------%
\begin{condition}
\label{mgipm:cond:cond1} The operators $\op{H}_{h}$ are symmetric with respect to $\innprd{\cdot}{\cdot}_h$, positive semidefinite, and
uniformly bounded with respect to $h\in I$, that is, there exists a constant $C(\op{H})>0$ independent on $h$ such that
\begin{equation}
\label{mgipm:cond:unifbdH}
\enorm{\op{H}_h}_h\le C(\op{H})\ ,\ \ \forall h\in I\ .
\end{equation}
Moreover, there exists $p > 0$ so that
\begin{equation}
\label{mgipm:cond:twolevapproxH}
\nnorm{(\op{H}_{h}-\op{H}_{2h}\pi_{2h})u} \le C(\op{H}) h^p \nnorm{u},\ \ \mathrm{for\ all\ }u\in\op{V}_h,\ h\in I\ .
\end{equation}
\end{condition}
\noindent For linear splines the optimal approximation order is $p=2$, but
for certain problems and discretizations the actual rate  may be suboptimal.
%%---------------------------------------%
%
%% TWO-GRID APPROXIMATION LEMMA
%
\begin{lemma}
\label{lma:GNapprox}
If $\op{H}_h$ satisfies Condition~\textnormal{\ref{mgipm:cond:cond1}}, then
\begin{equation}
\label{eq:GNapprox}
\nnorm{\op{G}_{h}-\op{N}_h} \le C(\op{H}) h^p\ .
\end{equation}
\end{lemma}
%
%%% BEGIN PROOF
\indent{\em Proof}.
This result is an immediate consequence of~\eqref{mgipm:cond:twolevapproxH}, since
$$\op{G}_{h}-\op{N}_h = (I+\op{H}_{h})-(\rho+ (I+\op{H}_{2h}) \pi) = \op{H}_{h} - \op{H}_{2h}\pi\ .\qquad\endproof$$
%%% END PROOF
%
%%---------------------------------------%
Verifying that~\eqref{mgipm:cond:twolevapproxH} of Condition~\ref{mgipm:cond:cond1} holds under
some general hypotheses is nontrivial for this problem due to the presence of multiple inner products that have to be taken into
consideration. More precisely, if $\op{L}_h^*$ were the dual of $\op{L}_h$ with respect to $\innprd{\cdot}{\cdot}$ instead of
$\innprd{\cdot}{\cdot}_h$, then Condition~\ref{mgipm:cond:cond1} would follow from the
approximability of $\op{L}$ by $\op{L}_h$ together with the smoothing properties of $\op{L}$, as is shown
in~\cite{MR2429872} (proof of Theorem 4.1). Hence a natural requirement is that $\innprd{\cdot}{\cdot}_h$  approximates
$\innprd{\cdot}{\cdot}$  well. 
%We restrict the remainder of this discussion to the case $d=2$.

%%---------------------------------------%
%
% LEMMA
%
\begin{lemma}
\label{lma:innprdapprox}
With $w_h$ chosen as in~\thref{eq:weighdef} there exists a constant $C=C(\op{T}_0)>0$  independent of $h$ such that
\begin{equation}
\label{eq:innprdtsapprox}
\abs{\innprd{u}{v}_h-\innprd{u}{v}}\le C h^2 \nnorm{u}_{H^1(\Omega)}\cdot\nnorm{v}_{H^1(\Omega)}\ ,\ \forall u,v \in \op{V}_h.
\end{equation}
\end{lemma}
%
% PROOF
%
%%% BEGIN PROOF
\indent{\em Proof}.
By Theorem 4.4.4 in~\cite{MR2373954}, given $d=2$,
there exists a constant $C>0$ that depends on the ``chunkiness'' of the initial triangulation $\op{T}_0$ but is 
independent of $h$ and of $T\in\op{T}_h$, so that for all $u, v\in \op{V}_h$
\begin{eqnarray*}
\nnorm{u v-\op{I}_h(u v)}_{L^1(T)}\le C h^2 \abs{u v}_{W^2_1(T)} \stackrel{u, v\ \mathrm{linear\ on\ }T}{=}
C h^2 \int_T \abs{\nabla u\cdot \nabla v}\ .
\end{eqnarray*}
After summing  over all triangles $T\in\op{T}_h$ we obtain
\begin{eqnarray*}
\nnorm{u v-\op{I}_h(u v)}_{L^1(\Omega)}\le C h^2  \nnorm{\nabla u\cdot \nabla v}_{L^1(\Omega)}
\le C h^2  \abs{u}_{H^1(\Omega)} \cdot \abs{v}_{H^1(\Omega)}\ .
\end{eqnarray*}
The conclusion follows from 
\begin{eqnarray*}
\abs{\innprd{u}{v}-\innprd{u}{v}_h} =  \Abs{\int_{\Omega}\left(u v - \op{I}_h(u v)\right)} \le\nnorm{u v-\op{I}_h(u v)}_{L^1(\Omega)}\ .\qquad\endproof
\end{eqnarray*}
%%---------------------------------------%

Throughout this section we denote by $C(\op{L})$ a generic constant that is proportional to the 
constant $C(\op{L})$ of Condition~\ref{mgipm:cond:condsmooth} if the proportionality depends only on the domain $\Omega$ and the initial
triangulation $\op{T}_0$.
\noindent We define the restriction operator \mbox{$\op{R}^w_{2h}:\op{V}_h\rightarrow \op{V}_{2h}$} by 
$$
\innprd{\iota_h u}{v}_h = \innprd{u}{\op{R}^w_{2h} v}_{2h}\ ,\ \forall u\in \op{V}_{2h}, v\in \op{V}_h\ ,
$$
where $\iota_h:\op{V}_{2h}\rightarrow \op{V}_h$ is the inclusion operator. It follows that $\op{R}^w_{2h}$ is
uniformly bounded with respect to $h\in I$, that is, there exists $C$ independent of $h\in I$ so that
\begin{equation}
\label{eq:Rstab}
\nnorm{\op{R}^w_{2h} u}\le C \nnorm u,\ \ \forall u\in\op{V}_h\ .
\end{equation}

We call a  triangulation $\op{T}$ \emph{locally symmetric} if for every vertex $P$ the union of triangles
in $\op{T}$ having $P$ as a corner is invariant with respect to the reflection through $P$ given by $r_P(x) = (2 P-x)$. If $\op{T}$ is locally symmetric
and $\varphi_P$ is the nodal basis function at $P$, then $\varphi_P\circ r_P = \varphi_P$. Furthermore, a simple calculation shows
that for any linear map $L(x)=a_1 x_1 + a_2 x_2$  we have
\begin{equation}
\label{eq:locsymmconseq}
\int_{\Omega}\varphi_P(x) L(x-P) dx = 0\ .
\end{equation}
Naturally, a uniform mesh is locally symmetric.

The following grouped results are either simple consequences
of Condition~\ref{mgipm:cond:condsmooth} or extracted from~\cite{MR2429872}.
%%---------------------------------------%
%
% LEMMA
%
\begin{lemma}
\label{lma:manyresults}
If $\op{L}, \op{L}_h$ satisfy Condition~\ref{mgipm:cond:condsmooth} there exist 
constants $C(\op{L})$ and \mbox{$C'=C'(\Omega)$}  independent of $h$ such that the following hold:\\
\textnormal{(a)} $H^1, L^2$ - uniform stability of $\op{L}_h$: 
\begin{equation}
\label{eq:H1stabdisc}
\nnorm{\op{L}_h u}_{H^m(\Omega)}\le C(\op{L}) \nnorm{u}\ ,\ \forall u\in \op{V}_h,\ m=0, 1;
\end{equation}
\textnormal{(b)} smoothing of negative-index norm: 
\begin{equation}
\label{eq:Hm2stabcont}
\nnorm{\op{L} u} \le C(\op{L}) \norm{u}_{\Hneg{m}}\ ,\ \forall u\in {\mathcal V}_h, m=1,2\ ;
\end{equation}
\textnormal{(c)}  negative-index norm approximation of the identity by $\pi_{2h}, \op{R}^w_{2h}$:
\begin{eqnarray}
\label{eq:projid_approx}
\norm{(I - \pi_{2h}) u}_{\Hneg{2}(\Omega)}& \le& C' h^2 \norm{u},\ \ \forall u\in \op{V}_h; \\
\label{eq:restr_approx}
\norm{(I - \op{R}^w_{2h}) u}_{\Hneg{p}(\Omega)}& \le& C' h^p \norm{u},\ \ \forall u\in \op{V}_h\ , 
\end{eqnarray}
where $p=1$ on an unstructured grid, and  $p=2$ on a locally symmetric grid;\\
%\footnote{The uniform, periodic grid is
%the three-line mesh obtained by dividing the unit square into equally sized squares with sides
%parallel to the coordinate axes, and by further cutting each little square along its slope-one
%diagonal; the functions are then considered doubly-periodic.};\\
%
\textnormal{(d)} $\op{L}$ diminishes high-frequencies:
\begin{eqnarray}
\label{eq:projid_approxL}
\norm{\op{L}(I - \pi_{2h}) u}& \le& C(\op{L}) h^2 \norm{u},\ \ \forall u\in \op{V}_h\ ; \\
\label{eq:restr_approxK}
\norm{\op{L}(I - \op{R}^w_{2h}) u}& \le& C(\op{L}) h^p \norm{u},\ \ \forall u\in \op{V}_h\ ,
\end{eqnarray}
where $p=1$ on an unstructured grid, and  $p=2$ on a locally symmetric grid;\\
\textnormal{(e)} 
\begin{eqnarray}
\label{eq:innprdKKH}
\abs{\innprd{\op{L} u}{\op{L} v}-\innprd{\op{L}_h u}{\op{L}_h v}} \le C(\op{L}) h^2 \nnorm{u}\cdot \nnorm{v}\ , \ \forall u\in \op{V}_h\ .
\end{eqnarray}
\end{lemma}
%
% PROOF
%
\begin{proof}
The stability conditions at (a) are direct consequences of~\eqref{mgipm:cond:par_smooth} and~\eqref{mgipm:cond:consist}, and (b) follows from~\eqref{mgipm:cond:par_smooth}
(see also~\cite{MR2429872} Corollary 6.2).
The estimate~\eqref{eq:projid_approx} is a straightforward consequence of the Bramble-Hilbert Lemma~\cite{MR2373954}, 
while \eqref{eq:restr_approx} follows from Theorem 6.6 in~\cite{MR2429872}
(see Example 6.7 for the uniform mesh case). The inequalities at (d) follow from (b) and (c), and (e) follows from~\eqref{mgipm:cond:consist} and the uniform boundedness of
$\op{L}_h$. 
\end{proof}
%%---------------------------------------%
%
% PROPOSITION
%
\begin{proposition}
\label{prop:conditionimplication}
If the operators $\op{L}$, $\op{L}_h$ satisfy Condition~\textnormal{\ref{mgipm:cond:condsmooth}} with the weights given by~\thref{eq:weighdef},
then Condition~\textnormal{\ref{mgipm:cond:cond1}} holds with $C(\op{H})=C(\op{L})$ and  $p=2$  if the meshes are locally symmetric, or $p=1$ otherwise.
\end{proposition}
%
% PROOF
%
\begin{proof} To simplify notation we write $\pi = \pi_{2h}, \op{R} = \op{R}^w_{2h}$.
First we have 
\begin{eqnarray*}
&&\innprd{\op{L}_{2h}^*\op{L}_{2h}\pi u}{v}_h = \innprd{\op{L}_{2h}^*\op{L}_{2h}\pi u}{\op{R} v}_{2h} = \innprd{\op{L}_{2h}\pi u}{\op{L}_{2h}\op{R} v}_{2h}\ ,\ \mathrm{and}\\
&&\innprd{\op{L}_h^*\op{L}_h u}{v}_h = \innprd{\op{L}_h u}{\op{L}_h v}_h\ .
\end{eqnarray*}
Therefore
\begin{eqnarray*}
  \lefteqn{\abs{\innprd{\op{L}_{2h}^*\op{L}_{2h}\pi u}{v}_h - \innprd{\op{L}_h^*\op{L}_h u}{v}_h} =
    \abs{\innprd{\op{L}_{2h}\pi u}{\op{L}_{2h}\op{R} v}_{2h}- \innprd{\op{L}_h u}{\op{L}_h v}_h}}\\
  &\le& \underbrace{\abs{\innprd{\op{L}_{2h}\pi u}{\op{L}_{2h}\op{R} v}_{2h}- \innprd{\op{L}_{2h}\pi u}{\op{L}_{2h}\op{R} v}}}_{A_1}+ 
  \underbrace{\abs{\innprd{\op{L}_{2h}\pi u}{\op{L}_{2h}\op{R} v}- \innprd{\op{L}\pi u}{\op{L}\op{R} v}}}_{A_2}\\
  &&+ \underbrace{\abs{\innprd{\op{L}\pi u}{\op{L}\op{R} v} - \innprd{\op{L} u}{\op{L} v}}}_{A_3}\\
  &&+\underbrace{\abs{\innprd{\op{L} u}{\op{L} v} - \innprd{\op{L}_h u}{\op{L}_h v}}}_{A_4} +
  \underbrace{\abs{\innprd{\op{L}_h u}{\op{L}_h v}-\innprd{\op{L}_h u}{\op{L}_h v}_h}}_{A_5}\ .
\end{eqnarray*}
For $A_1$ we have
\begin{eqnarray*}
A_1 &\stackrel{\eqref{eq:innprdtsapprox}}{\le}& C(\op{L}) h^2 \nnorm{\op{L}_{2h}\pi u}_{H^1(\Omega)}\cdot \nnorm{\op{L}_{2h}\op{R} v}_{H^1(\Omega)} 
\stackrel{\eqref{eq:H1stabdisc}}{\le} C(\op{L}) h^2 \nnorm{\pi u}\cdot \nnorm{\op{R} v}\\
&\stackrel{\eqref{eq:Rstab}}{\le}&C(\op{L}) h^2 \nnorm{u}\cdot \nnorm{v}\ ,
\end{eqnarray*}
and a similar estimate  holds for $A_5$. Also~\eqref{eq:innprdKKH} implies that 
$$\max(A_2, A_4) \le C(\op{L}) h^2 \nnorm{u}\cdot \nnorm{v}\ .$$
For $A_3$ we have
\begin{eqnarray*}
A_3&\le& \abs{\innprd{\op{L}(\pi - I)u}{\op{L}\op{R} v}}+\abs{\innprd{\op{L} u}{(\op{L}(\op{R} - I)v)}} 
\stackrel{\eqref{eq:projid_approxL},~\eqref{eq:restr_approxK}}{\le} C(\op{L}) h^p \nnorm{u}\cdot \nnorm{v}\ .
\end{eqnarray*}
Since $p\le 2$, $A_3$ is the weak link, and we have
$$\abs{\innprd{(\op{H}_{2h}\pi -\op{H}_h)u}{v}_h} = \abs{\innprd{(\op{L}_{2h}^*\op{L}_{2h}\pi -\op{L}_h^*\op{L}_h)u}{v}_h}
\le C(\op{L}) h^p \nnorm{u} \cdot \nnorm{v},\ \forall v\in \op{V}_h\ ,$$
and the conclusion follows from the equivalence of $\nnorm{\cdot}$ with $\enorm{\cdot}_h$.
\end{proof}\\
For $\rho\in W^2_{\infty}(\Omega)$ we denote by
$$
\nnorm{\rho}_{W^2_{\infty}(\Omega)/\mathbb{R}} = \max_{1\le \abs{\alpha}\le 2} \nnorm{\partial^{\alpha} \rho}_{L^{\infty}(\Omega)}\ ,
$$
which effectively is the norm on the quotient space $W^2_{\infty}(\Omega)/\mathbb{R}$.

%%---------------------------------------%
%% LEMMA
\begin{lemma} 
\label{mgipm:lma:dhapprox}
If $\rho\in W^2_{\infty}(\Omega)$ then there exists a constant $C>0$ 
independent of $\rho, h$ so that
\begin{eqnarray}
\label{mgipm:eq:dhapprox}
\nnorm{(\op{D}_{\rho}-\op{D}^h_{\rho_h}) u} \le C h^2 \nnorm{\rho}_{W^2_{\infty}(\Omega)/\mathbb{R}} \nnorm{u}_{H^1(\Omega)}\ ,\ \ \forall u\in\op{V}_h,\  h\in I\ ,
\end{eqnarray}
where $\rho_h = \op{I}_h(\rho)$.
\end{lemma}\\
%% PROOF
\indent{\em Proof}.
Note first that $\op{D}^h_{\rho} =  \op{D}^h_{\rho_h} = \op{I}_h \op{D}_{\rho}$, since only the node values
of $\rho$ enter the definition of the operator $\op{D}^h_{\rho}$. Given $u\in \op{V}_h$, for each triangle $T\in \op{T}_h$ we have
\begin{equation}
\label{mgipm:eq:2normeval}
\abs{\rho \:u}_{H^2(T)}\le C\nnorm{\rho}_{W^2_{\infty}(T)/\mathbb{R}}\cdot\norm{u}_{H^1(T)}\ ,
\end{equation}
because $u$ is linear on $T$. Therefore
\begin{eqnarray*}
\norm{(\op{D}_{\rho}-\op{D}^h_{\rho}) u}&=& \nnorm{\rho\: u - \op{I}_h( \rho\: u)} 
\le  C h^2 \left(\sum_{T\in \op{T}_h}\abs{\rho\: u}_{H^2(T)}^2\right)^{\frac{1}{2}} \\
&\le &C h^2 \left(\sum_{T\in \op{T}_h}\nnorm{\rho}_{W^2_{\infty}(T)/\mathbb{R}}^2 \nnorm{u}_{H^1(T)}^2\right)^{\frac{1}{2}}
\le C h^2 \nnorm{\rho}_{W^2_{\infty}(\Omega)/\mathbb{R}} \nnorm{u}_{H^1(\Omega)}.\qquad\endproof
\end{eqnarray*}

%%---------------------------------------%
%% LEMMA
\begin{lemma} 
\label{mgipm:lma:dhapproxhm1}
If $\rho\in W^2_{\infty}(\Omega)$ then there exists a constant $C>0$ 
independent of $\rho, h$ so that
\begin{eqnarray}
\label{mgipm:eq:dhapproxhm1}
\nnorm{(\op{D}_{\rho}-\op{D}^h_{\rho_h}) u}_{\widetilde{H}^{-1}(\Omega)} \le C h^p \nnorm{\rho}_{W^p_{\infty}(\Omega)/\mathbb{R}} \nnorm{u}\ ,\ \ 
\forall u\in\op{V}_h,\  h\in I\ ,
\end{eqnarray}
where $\rho_h = \op{I}_h(\rho)$, and $p=2$ if the mesh is locally symmetric, otherwise $p=1$.
\end{lemma}
%% PROOF
\begin{proof} 
We focus on the situation when the mesh is locally symmetric and leave the general case as an exercise.
Let $u\in\op{V}_h, v\in H_0^1(\Omega)$ be arbitrary, and denote 
by $S_i=\mathrm{supp}(\varphi_i^h)$. The constant $C$ is assumed to be independent of $u, v, \rho, h$.
First remark that each triangle in $\op{T}_h$ lies in at most
three of the sets $S_i$ and that
\begin{equation}
\label{eq:diamSi}
\mathrm{diam}(S_i) \le C h,\ \ 1\le i\le N_h\ 
\end{equation}
due to the quasi-uniformity for the meshes. Also note that 
\begin{equation}
\label{eq:interpidentprd}
\rho \varphi_i^h - \op{I}_h(\rho \varphi_i^h) = \varphi_i^h(\rho-\rho(P_i^h)),\ \ 1\le i\le N_h\ .
\end{equation}
Further we define $v_i = \frac{1}{\mathrm{area}(S_i)}\int_{S_i} v$, for $1\le i\le N_h$,
and $u_i=u(P_i^h)$. Then \mbox{$u=\sum_{i=1}^{N_h} u_i \varphi^h_i$} and 
\begin{eqnarray*}
{\Abs{\innprd{(\op{D}_{\rho}-\op{D}^h_{\rho_h}) u}{v}} =  \Abs{\int_{\Omega}\left(\rho u-\op{I}_h(\rho u)\right) v} 
\stackrel{\eqref{eq:interpidentprd}}{=}
\Abs{\sum_{i=1}^{N_h} u_i \left(\int_{S_i}\varphi_i^h (\rho-\rho(P_i^h)) v\right)}}\\
\le
\sum_{i=1}^{N_h} \abs{u_i} \left(\Abs{\int_{S_i}\varphi_i^h (\rho-\rho(P_i^h)) (v-v_i)} + \Abs{\int_{S_i}\varphi_i^h (\rho-\rho(P_i^h)) v_i\right)}\ .
\end{eqnarray*}
For the first term in the sum above
\begin{eqnarray*}
\Abs{\int_{S_i}\varphi_i^h (\rho-\rho(P_i^h)) (v-v_i)}&
\stackrel{}{\le}& \nnorm{\varphi_i^h}\cdot \nnorm{\rho-\rho(P_i^h)}_{L^{\infty}(S_i)}\cdot \nnorm{v-v_i}_{L^2(S_i)}\\
&\stackrel{\eqref{eq:diamSi}}{\le}& C h^2 \nnorm{\varphi_i^h}\cdot \abs{\rho}_{W^1_{\infty}(S_i)}\cdot \abs{v}_{H^1(S_i)}\ .
\end{eqnarray*}
For the second term in the sum we take advantage of the local grid symmetry:
\begin{eqnarray*}
\Abs{\int_{S_i}\varphi_i^h (\rho-\rho(P_i^h)) v_i}& \stackrel{\eqref{eq:locsymmconseq}}{=} &\Abs{\int_{S_i}v_i\: \varphi_i^h(x) (\rho(x)-\rho(P_i^h)-d\rho_{P_i^h}(x-P_i^h))dx}\\
&\le& \nnorm{v_i\: \varphi_i^h}_{L^1(S_i)}\cdot \nnorm{\rho-\rho(P_i^h)-d\rho_{P_i^h}(x-P_i^h)}_{L^{\infty}(S_i)}\\
&\le& C h^2 \nnorm{v_i}_{L^2(S_i)} \nnorm{\varphi_i^h} \cdot \abs{\rho}_{W^2_{\infty}(S_i)}\ .
\end{eqnarray*}
Since $\norm{v_i}_{L^2(S_i)} \le \norm{v}_{L^2(S_i)}$ we now have
\begin{eqnarray*}
\lefteqn{\Abs{\int_{\Omega}\left(\rho u-\op{I}_h(\rho u)\right) v}} \\
&\le& C h^2 \sum_{i=1}^{N_h} \abs{u_i} \nnorm{\varphi_i^h} 
\left(\abs{\rho}_{W^1_{\infty}(S_i)}\cdot \abs{v}_{H^1(S_i)} + \nnorm{v}_{L^2(S_i)} \abs{\rho}_{W^2_{\infty}(S_i)}\right)\\
&\le & C h^2 \nnorm{\rho}_{W^2_{\infty}(\Omega)/\mathbb{R}}\sum_{i=1}^{N_h} \abs{u_i}\cdot \nnorm{\varphi_i^h} \cdot \nnorm{v}_{H^1(S_i)}\\
&\le & C h^2 \nnorm{\rho}_{W^2_{\infty}(\Omega)/\mathbb{R}}\left(\sum_{i=1}^{N_h} \abs{u_i}^2\cdot \nnorm{\varphi_i^h}^2 \right)^{\frac{1}{2}}
\cdot \left(\sum_{i=1}^{N_h} \nnorm{v}^2_{H^1(S_i)}\right)^{\frac{1}{2}}\\
&\le& C  h^2 \nnorm{\rho}_{W^2_{\infty}(\Omega)/\mathbb{R}} \nnorm{u}\cdot \nnorm{v}_{H^1(\Omega)}\ ,
\end{eqnarray*}
where for the last inequality we used the quasi-uniformity of the mesh and the fact the each triagle is in at most three of the sets $S_i$.
The conclusion follows after dividing by $\nnorm{v}_{H^1(\Omega)}$ and taking the $\sup$ over all $v\in H_0^1(\Omega)$.
\end{proof}

%%---------------------------------------%
%
% PROPOSITION
%
\begin{proposition}
\label{prop:conditionimplicationKL}
If the operators $\op{K}$, $\op{K}_h$ satisfy Condition~\textnormal{\ref{mgipm:cond:condsmooth}} with $p=2$ on the locally-symmetric
meshes $\op{T}_h$ with the weights given by~\thref{eq:weighdef}, then $\op{L}$, $\op{L}_h$ also satisfy Condition~\textnormal{\ref{mgipm:cond:condsmooth}} with 
$$C(\op{L})= C(\op{K}) \nnorm{\lambda^{-\frac{1}{2}}}_{W^2_{\infty}(\Omega)}\ .$$ If the meshes are not locally-symmetric, then the power of $h$ in 
Condition~\textnormal{\ref{mgipm:cond:condsmooth}} \textnormal{[b]} for the operators $\op{L}, \op{L}_h$ is $1$ for both $m=0, 1$. 
\end{proposition}\\
%
% PROOF
%
\indent{\em Proof}.
Note that for $\rho\in L^{\infty}(\Omega), \nnorm{\op{D}_{\rho }u} \le \nnorm{\rho}_{L^{\infty}(\Omega)}\cdot \nnorm{u}$,
therefore
\begin{eqnarray*}
\nnorm{\op{K} \op{D}_{1/\sqrt{\lambda}} u}_{H^m(\Omega)} \le C(\op{K}) \nnorm{\op{D}_{1/\sqrt{\lambda}} u}\le 
C(\op{K}) \nnorm{\lambda^{-\frac{1}{2}}}_{L^{\infty}(\Omega)}\cdot \nnorm{u}\ ,\ \ m=0, 1, 2,
\end{eqnarray*}
which implies the smoothing condition~\eqref{mgipm:cond:par_smooth}.
For $u\in \op{V}_h$
\begin{eqnarray*}
\lefteqn{\nnorm{(\op{K} \op{D}_{1/\sqrt{\lambda}} - \op{K}_h \op{D}^h_{1/\sqrt{\lambda}}) u}_{H^1(\Omega)}}\\
& \le&
\nnorm{\op{K} (\op{D}_{1/\sqrt{\lambda}}-\op{D}^h_{1/\sqrt{\lambda}})u}_{H^1(\Omega)} + \nnorm{(\op{K} - \op{K}_h) \op{D}^h_{1/\sqrt{\lambda}} u}_{H^1(\Omega)}\\
&\stackrel{\eqref{mgipm:cond:par_smooth},\eqref{mgipm:cond:consist}}{\le}&
C(\op{K})\nnorm{(\op{D}_{1/\sqrt{\lambda}}-\op{D}^h_{1/\sqrt{\lambda}})u} + C(\op{K}) h\nnorm{\op{D}^h_{1/\sqrt{\lambda}} u}\\
&\stackrel{\eqref{mgipm:eq:dhapprox}}{\le}&
C(\op{K}) h^2 \nnorm{\lambda^{-\frac{1}{2}}}_{W^2_{\infty}(\Omega)/\mathbb{R}} \nnorm{u}_{H^1(\Omega)} + 
C(\op{K}) h\nnorm{\lambda^{-\frac{1}{2}}}_{L^{\infty}(\Omega)} \nnorm{ u}\\
&\le& C(\op{K}) h\nnorm{\lambda^{-\frac{1}{2}}}_{W^2_{\infty}(\Omega)} \nnorm{ u}\ ,
\end{eqnarray*}
where in the last inequality we have used an inverse estimate. This proves~\eqref{mgipm:cond:consist} for $\op{K}, \op{K}_h$ with $m=1$ and concludes 
the proof for the non-locally-symmetric case. 
For $m=2$ and locally-symmetric mesh
\begin{eqnarray*}
\lefteqn{\nnorm{(\op{K} \op{D}_{1/\sqrt{\lambda}} - \op{K}_h \op{D}^h_{1/\sqrt{\lambda}}) u}}\\
& \le&
\nnorm{\op{K} (\op{D}_{1/\sqrt{\lambda}}-\op{D}^h_{1/\sqrt{\lambda}})u} + \nnorm{(\op{K} - \op{K}_h) \op{D}^h_{1/\sqrt{\lambda}} u}\\
&\stackrel{\eqref{eq:Hm2stabcont},\eqref{mgipm:cond:consist}}{\le}&
C(\op{K})\nnorm{(\op{D}_{1/\sqrt{\lambda}}-\op{D}^h_{1/\sqrt{\lambda}})u}_{\Hneg{1}(\Omega)} + C(\op{K}) h^2\nnorm{\op{D}^h_{1/\sqrt{\lambda}} u}\\
&\stackrel{\eqref{mgipm:eq:dhapproxhm1}}{\le}&
C(\op{K}) h^2 \nnorm{\lambda^{-\frac{1}{2}}}_{W^2_{\infty}(\Omega)/\mathbb{R}} \nnorm{u} + C(\op{K}) h^2\nnorm{\lambda^{-\frac{1}{2}}}_{L^{\infty}(\Omega)} \nnorm{ u}\\
&\le& C(\op{K}) h^2\nnorm{\lambda^{-\frac{1}{2}}}_{W^2_{\infty}(\Omega)} \nnorm{ u}.\qquad\endproof
\end{eqnarray*}

We conclude the analysis of the two-grid preconditioner with the computation of the \emph{spectral distance} (SD) $d_{\sigma}$ 
between $(\op{G}_h)^{-1}$ and the two-grid preconditioner $\op{S}_h$ defined in~\eqref{mgipm:eq:twogridprecinv}. This step facilitates
a smooth transition from the two-grid analysis to the multigrid analysis of the next section.
Given a Hilbert space $(\op{X}, \innprd{\cdot}{\cdot})$ we denote by $\mathfrak{L}_+(\op{X})$ the set of operators with positive definite symmetric part:
$$ \mathfrak{L}_+(\op{X})= \left\{T\in \mathfrak{L}(\op{X}): \innprd{T u}{u} > 0,\ \ \forall u\in
\op{X}\setminus\{0\}\right\}\ .$$  
First we define the joined numerical range of $S, T\in\mathfrak{L}_+(\op{X})$ by
\begin{eqnarray*}
W(S, T) = \left\{ \frac{\innprd{S_{\mathbb C} w}{w}}{\innprd{T_{\mathbb C} w}{w}}\ :\ w \in \op{X}^{\mathbb{C}}\setminus \{0\}\right\}\ ,
\end{eqnarray*}
where $T_{\mathbb C}(u+{\bf i}v) = T(u)+{\bf i}T(v)$ is the complexification of $T$.
Note that if $T$ is symmetric positive definite, then $W(S, T)$ is simply the numerical range of $T^{-\frac{1}{2}} S T^{-\frac{1}{2}}$.
The spectral distance between $S, T \in \mathfrak{L}_+(\op{X})$, introduced in~\cite{MR2429872} as a measure of spectral equivalence between $S$ and $T$, 
is defined by
\begin{eqnarray*}
d_{\sigma}(S, T)& =& \sup\{\abs{\ln z}\ :\  z \in W(S, T) \}\ ,
\end{eqnarray*}
where $\ln$ is the branch of the logarithm corresponding to 
$\mathbb{C}\setminus (-\infty, 0]$.
Following Lemma 3.2 in~\cite{MR2429872}, if $W(S,T)\subseteq \op{B}_{\alpha}(1) = \{z\in\mathbb{C}\ :\ \abs{z-1} < \alpha\}$ with $\alpha\in (0,1)$, then
\begin{equation}
\label{eq:distineq}
d_{\sigma}(S, T)\le \frac{\abs{\ln (1-\alpha)}}{\alpha} \sup\{|z-1|\  : \ z\in W(S,T)\}\ ,
\end{equation}
which offers a practical way to estimate the spectral distance when it is small.
The spectral distance serves 
both as a means to quantify the quality of a preconditioner
and also as a convenient analysis tool for multigrid algorithms. Essentially, if two operators $S, T$ satisfy
$$1-\delta\le \Abs{\frac{\innprd{S_{\mathbb C} w}{w}}{\innprd{T_{\mathbb C} w}{w}}}\le 1+ \delta\ ,\  \ \forall w \in \op{X}^{\mathbb{C}}\setminus \{0\},$$
with $\delta \ll 1$, then $d_{\sigma}(S, T) \approx \delta$. 
If $N\approx G^{-1}$ is a preconditioner for $G$, then both $d_{\sigma}(N, G^{-1})$ and $d_{\sigma}(N^{-1}, G)$ (quantities which are
are equal if $G, N$ are symmetric) are shown to control the spectral radius $\rho(I-N G)$ (see Lemma~\ref{lma:dist_ctrl_radius} in   Appendix~\ref{sec:spec_dist} 
for a precise formulation), which is an accepted quality-measure for a preconditioner. The advantage of using $d_{\sigma}$ 
over $\rho(I-N G)$ is that the former is a true distance function.
%
%% TWO-GRID APPROXIMATION THEOREM
%
\begin{theorem} 
\label{mgipm:th:twogrid}
If the operators $\op{K}$, $\op{K}_h$ satisfy Condition~\textnormal{\ref{mgipm:cond:condsmooth}} on the locally-symmetric
meshes $\op{T}_h$ with the weights given by~\thref{eq:weighdef}, and $\lambda\in W^2_{\infty}(\Omega)$ satisfies 
\mbox{$\op{I}_h \lambda = \lambda_h$},  there exist $C, \delta>0$  independent of $h$, $\lambda$ so that for 
\mbox{$h^2 \nnorm{\lambda^{-\frac{1}{2}}}_{W^2_{\infty}(\Omega)} \le \delta$} 
\begin{equation}
\label{mgipm:eq:NGest}
d_{\sigma}\left(\op{G}_h^{-1}, \op{S}_h\right) \le C h^2 \nnorm{\lambda^{-\frac{1}{2}}}_{W^2_{\infty}(\Omega)}\ ,
\end{equation}
where $\op{G}_h$  and $\op{S}_h$ are defined as in~\thref{mgipm:eq:fineopdef} and~\thref{mgipm:eq:twogridprecinv}. 
If the meshes are not locally-symmetric then the power of $h$ in\textnormal{~\eqref{mgipm:eq:NGest}} is $1$.
\end{theorem}\\
%%% BEGIN PROOF
\indent{\em Proof}.
Again, we restrict our attention to the locally-symmetric case.
The operator $\op{G}_h$ is symmetric positive definite with respect $\innprd{\cdot}{\cdot}_h$ and satisfies 
$$\innprd{\op{G}_h u}{u}_h = \enorm{u}_h^2 + \enorm{\op{L}_h u}_h^2 \ge  \enorm{u}_h^2\ .$$
Therefore $\sigma(\op{G}_h)\subseteq [1,\infty)$, and $\enorm{\op{G}_h^{\nu}}_h \le 1$ for all $\nu<0$. Due to the norm equivalence 
$\enorm{\cdot}_h \sim \nnorm{\cdot}$ there exists $C_1>0$ so that $\nnorm{\op{G}_h^{-\frac{1}{2}}} \le C_1$.
By Lemma~\ref{lma:GNapprox} and Propositions~\ref{prop:conditionimplication} and~\ref{prop:conditionimplicationKL}
we have
\begin{eqnarray*}
\nnorm{\op{G}_h- \op{N}_h}\le C_2 h^2 \nnorm{\lambda^{-\frac{1}{2}}}_{W^2_{\infty}(\Omega)}
\end{eqnarray*}
for some constant $C_2>0$. Hence for $C_3=C_1^2 C_2$
\begin{eqnarray*}
\nnorm{I-\op{G}_h^{-\frac{1}{2}} \op{N}_h \op{G}_h^{-\frac{1}{2}}}\le \nnorm{\op{G}_h^{-\frac{1}{2}}}^2\cdot \nnorm{\op{G}_h- \op{N}_h} 
\le  C_3 h^2 \nnorm{\lambda^{-\frac{1}{2}}}_{W^2_{\infty}(\Omega)}\ .
\end{eqnarray*}
Since for any operator $T$ we have $\nnorm{I-T^{-1}}\le (1-\alpha)^{-1} \nnorm{I-T}$ if $\nnorm{I-T} \le \alpha$, 
\begin{eqnarray}
\label{eq:ineq13}
\nnorm{I-\op{G}_h^{\frac{1}{2}} \op{S}_h \op{G}_h^{\frac{1}{2}}}\stackrel{\op{S}_h=\op{N}_h^{-1}}{\le}
\frac{4}{3}\nnorm{I-\op{G}_h^{-\frac{1}{2}} \op{N}_h \op{G}_h^{-\frac{1}{2}}}{\le} 
\frac{4 C_3}{3} h^2 \nnorm{\lambda^{-\frac{1}{2}}}_{W^2_{\infty}(\Omega)}\ ,
\end{eqnarray}
provided $h^2 \nnorm{\lambda^{-\frac{1}{2}}}_{W^2_{\infty}(\Omega)}\le \delta=\frac{1}{4 C_3}$.
By further restricting $\delta$, we can assume the right-hand side of~\eqref{eq:ineq13} to be $\le 1/2$,
which implies that $W(\op{S}_h, \op{G}_h^{-1}) \subset \op{B}_{\frac{1}{2}}(1)$. By~\eqref{eq:distineq}
we obtain
\begin{eqnarray*}
d_{\sigma}(\op{S}_h, \op{G}_h^{-1})\le \frac{8 \ln 2}{3} C_3 h^2 \nnorm{\lambda^{-\frac{1}{2}}}_{W^2_{\infty}(\Omega)}\ .\qquad\endproof
\end{eqnarray*}
%%% END PROOF

The highlight of the last result is the presence of $O(h^2)$ (or $O(h)$ for general quasi-uniform meshes)
in the right-hand side of~\eqref{mgipm:eq:NGest}, which is the optimal order of approximation in $h$. 
We should stress that for classical multigrid methods for differential equations one has $O(1)$ as the right-hand side estimate, which is sufficient for
mesh-independence. In this case, if the theoretically introduced smooth function $\lambda$ could be the same for all meshes, 
the number of  $\op{S}_h$-preconditioned iterations is expected to {\bf decrease} with $h\downarrow 0$.
In reality, the discrete function $\lambda_h$  is tied to the Lagrange multipliers $\vect{v}_1, \vect{v}_2$, which in turn are 
related (actually expected to converge to as $\mu, h\downarrow 0$) the Lagrange multipliers  
$\underline{\lambda}, \overline{\lambda}$ of the continuous problem. Since in general the latter are only in $L^2$, 
 the factor $\nnorm{\lambda^{-\frac{1}{2}}}_{W^2_{\infty}(\Omega)}$ is expected to be unbounded as $\mu\downarrow 0$.
Therefore the preconditioning qualities of $\op{S}_h$ are expected to increase with $h\downarrow 0$, but decrease with $\mu\downarrow 0$. 
Thus for large-scale, high-resolution problems, where $h\ll 1$, the presented method is expected to perform very well, 
especially in connection with the multigrid method discussed in the next section.
However, for fixed $h$, as $\mu\downarrow 0$ in the IPM formulation and the approximate solution approaches $\widehat{\vect{u}}$, if the 
inequality constraints are active
then the quality of the proposed preconditioner normally degrades.
The advantages or disadvantages of this method will ultimately be discussed based on numerical experiments in Section~\ref{mgimp:sec:num_examples}.

%-----------SECTION-----------%
\section{The multigrid preconditioner} 
\label{mgipm:sec:multigrid}
While the  two-grid preconditioner $\op{S}_h$ may be efficient in terms of number of iterations, 
it is expensive to apply. In this section we develop a { multigrid} preconditioner $\op{S}^{mg}_h$ that also satisfies the optimal
order estimate~\eqref{mgipm:eq:NGest} but has a lesser cost. Since the process of passing from a two-grid to a multigrid
preconditioner of comparable quality has been  analyzed in~\cite{MR2429872}, we give
here only a brief description. In this section we assume a finite number of grids
$$I_{\max}=\{h_i\}_{0\le i\le i_{\max}},\ h_i=h_0 2^{-i}, $$
and the goal is to ultimately construct an efficient multigrid preconditioner for the operator on the finest grid.

Consider the operator
$\mathfrak{I}_{i-1}^i:\mathfrak{L}(\op{V}_{h_{i-1}})\rightarrow \mathfrak{L}(\op{V}_{h_i})$
by
$$
\mathfrak{I}_{i-1}^i (\op{M})\stackrel{\mathrm{def}}{=} (I-\pi_{h_{i-1}})+\op{M}\pi_{h_{i-1}}\ .$$
Cf.~\eqref{mgipm:eq:twogridprecinv} we have $\op{S}_{h_i} = \mathfrak{I}_{i-1}^i  (\op{G}^{-1}_{h_{i-1}})$. 
If we define $\op{S}^{V}_h$, $h\in I_{\max}$,  recursively by
\begin{equation}
\label{eq:Vcycledef}
\op{S}^{V}_{h_i} = \left\{\begin{array}{lll}\vspace{7pt}
 \op{G}^{-1}_{h_0}&\mathrm{if}& i=0\ ,\\
 \mathfrak{I}_{i-1}^i  (\op{S}^{V}_{h_{i-1}})&\mathrm{if}& 1\le i\le i_{\max}\ , \end{array}\right .
\end{equation}
then $\op{S}^{V}_h$ has a V-cycle structure. However, it is shown in~\cite{MR2429872} that $\op{S}^{V}_h$ is suboptimal,
in that it satisfies~\eqref{mgipm:eq:NGest} with $h^2$ replaced by $h^2_0$. Thus the quality of $\op{S}^{V}_h$ does not
improve with $h\downarrow 0$, as desired, it is simply mesh-independent (it only depends on $h_0$). To achieve the
desired result we  define the operator $\mathfrak{N}_i:\mathfrak{L}(\op{V}_{h_i})\rightarrow \mathfrak{L}(\op{V}_{h_i})$
$$
\mathfrak{N}_i (\op{M})\stackrel{\mathrm{def}}{=} 2\op{M} - \op{M}\: \op{G}_{h_i}\:\op{M}\ .$$
The latter is related to Newton's method for the operator-equation \mbox{$\op{X}^{-1}-\op{G}_{h_i}=0$};
namely, if $\op{X}_0$ is a good guess at the solution, i.e., approximates well $\op{G}^{-1}_{h_i}$, then the
first Newton iterate starting at $\op{X}_0$ is $\op{X}_1=\mathfrak{N}_i(\op{X}_0)$ (see also Remark 3.11 in~\cite{MR2429872}).
We define the multigrid preconditioner using the following algorithm:

\vspace{6pt}
\noindent{\bf Algorithm 1:}\emph{ Operator-form definition of $\op{S}^{mg}_{h_i}$}
\vspace{6pt}
\begin{enumerate}
\item {\tt if}\hspace{3pt} $i=0$ \hspace{3pt} 
\vspace{6pt}
\item \hspace{20pt} $\op{S}^{mg}_{h_0} := \op{G}^{-1}_{h_0}$ \hspace{96pt} {\tt \%} coarsest level
\vspace{6pt}
\item {\tt else if} \hspace{3pt} $i<i_{\max}$ \hspace{10pt} 
\vspace{6pt}
\item \hspace{42pt} $\op{S}^{mg}_{h_i} :=\mathfrak{N}_i(\mathfrak{I}_{i-1}^i (\op{S}^{mg}_{h_{i-1}}))$ \hspace{22pt} {\tt \%} intermediate level
\vspace{6pt}
\item  \hspace{20pt} {\tt else}\hspace{3pt} 
\vspace{6pt}
\item \hspace{42pt} $\op{S}^{mg}_{h_i} :=\mathfrak{I}_{i-1}^i (\op{S}^{mg}_{h_{i-1}})$ \hspace{42pt} {\tt \%} finest level
\vspace{6pt}
\item  \hspace{20pt} {\tt end if} 
\vspace{6pt}
\item  {\tt end if} 
\vspace{6pt}
\end{enumerate}
The key factor in Algorithm 1 is the application of $\mathfrak{N}_i$ at Step 4, and here is why: while $\op{G}^{-1}_{h_i}$
is well approximated by  $\mathfrak{I}_{i-1}^i(\op{S}^{mg}_{h_{i-1}})$ provided that  $\op{G}^{-1}_{h_{i-1}}\approx \op{S}^{mg}_{h_{i-1}}$ (recall
that $\op{G}^{-1}_{h_i}{\approx} \op{S}_{h_i} = \mathfrak{I}(\op{G}^{-1}_{h_{i-1}})$), an application
of $\mathfrak{N}_i$ brings $\mathfrak{I}(\op{S}^{mg}_{h_{i-1}})$ even closer to $\op{G}^{-1}_{h_i}$. This step is critical if
we want $d_{\sigma}(\op{G}^{-1}_{h}, \op{S}^{mg}_{h})=O(d_{\sigma}(\op{G}^{-1}_{h}, \op{S}_{h}))$.
Also, there are two main reasons for splitting the cases of intermediate vs. finest resolution, as opposed to just replacing 
$\mathfrak{I}_{i-1}^i (\op{S}^{V}_{h_{i-1}})$ with $\mathfrak{N}_i(\mathfrak{I}_{i-1}^i (\op{S}^{V}_{h_{i-1}}))$ in~\eqref{eq:Vcycledef}. 
First we would like to have
$\op{S}^{mg}_h=\op{S}_h$ for $h=h_{i_{\max}}$ if only two grids are used. Second, the application of $\mathfrak{N}_i$ includes a multiplication 
by $\op{G}_{h_i}$; since for the intended large-scale applications the finest-level mat-vec is expected to be 
very costly, we prefer that no such mat-vecs are computed  {\bf inside} the preconditioner. 
\begin{theorem} 
\label{mgipm:th:multigrid}
In the hypotheses of Theorem~\textnormal{\ref{mgipm:th:twogrid}}, and with $\op{S}^{mg}_h$ 
defined as in Algorithm {\textnormal 1} there exist $C, \delta>0$ 
independent of $h$ and $\lambda$ so that for $h^2 \nnorm{\lambda^{-\frac{1}{2}}}_{W^2_{\infty}(\Omega)} \le \delta$
\begin{equation}
\label{mgipm:eq:GSmg}
d_{\sigma}\left(\op{G}_h^{-1}, \op{S}^{mg}_h\right) \le C h^2 \nnorm{\lambda^{-\frac{1}{2}}}_{W^2_{\infty}(\Omega)}\ ,\  \mathrm{for}\ h=h_{\max}\ .
\end{equation}
\end{theorem}
\noindent The proof of Theorem~\ref{mgipm:th:multigrid} follows closely that of Theorem 5.4 in~\cite{MR2429872} and, in the interest 
of brevity, we do not give further details. Suffice it to say that the use of the spectral distance is instrumental, and that
an essential ingredient is the symmetry (with respect to $\innprd{\cdot}{\cdot}_h$) of $\op{G}_h$.

In practice, for large-scale problems, neither $\op{G}_h$ nor $\op{S}^{mg}_h$ are ever formed, so both are to be applied matrix-free. 
A simple verification shows that, given  some $\op{M}\in\mathfrak{L}(\op{V}_{h_i})$, the vector $\tilde{u}=(\mathfrak{N}_i(\op{M}))r$ can be computed by
setting $\tilde{u}:=u_2$ where $u_{k+1}:=u_k+\op{M}(r-\op{G}_{h_i}u_k)$ with $u_0=0$. Thus the matrix-free application of $\op{S}^{mg}_h$ is computed 
by the following function:

\vspace{6pt}
\noindent{\bf Algorithm 2:} Matrix-free implementation of the action $u=\op{S}^{mg}_h r$.
\vspace{6pt}
\begin{enumerate}
\item {\tt function}\ $u = MG(r, i)$
\vspace{6pt}
\item {\tt if}\hspace{3pt} $i=0$ \hspace{3pt} \hspace{100pt} {\tt \%} coarsest level
\vspace{6pt}
\item \hspace{20pt} $u := \op{G}^{-1}_{h_0} r$ \hspace{78pt} {\tt \%} direct or unpreconditioned CG solve
\vspace{6pt}
\item {\tt else} 
\vspace{6pt}
\item \hspace{20pt} $u := (I-\pi_{h_{i-1}}) r + MG(\pi_{h_{i-1}} r, i-1)$
\vspace{6pt}
\item \hspace{20pt}{\tt if} $i<i_{\max}$ \hspace{78pt} {\tt \%} intermediate level
\vspace{6pt}
\item \hspace{40pt} $r_1:=r-\op{G}_{h_i}u$
\vspace{6pt}
\item \hspace{40pt} $u_1 := (I-\pi_{h_{i-1}}) r_1 + MG(\pi_{h_{i-1}} r_1, i-1)$
\vspace{6pt}
\item \hspace{40pt} $u := u+ u_1$
\vspace{6pt}
\item \hspace{20pt} {\tt end if}\hspace{3pt} 
\vspace{6pt}
\item {\tt end if}\hspace{3pt} 
\vspace{6pt}
\end{enumerate}
As can be readily seen, Algorithm 2 has a W-cycle structure. To estimate the cost of $MG(\cdot, i_{\max})$
we denote by $T(i)$ the cost of applying $MG(\cdot, i)$ for $0<i<i_{\max}$. If we assume that 
one residual computation at level $i$ has complexity $O(N_{h_i}^2)$, and that the cost of computing an $L^2$-projection is negligible 
compared to that of a residual computation (this is reasonable for most applications since mass matrices are normally easy to invert)
then the resulting recursion for the function $T$ reads:
$$
T(i)=O(N_{h_i}^2) + 2 T(i-1)\ .
$$
For $i=i_{\max}$ the term $O(N_{h_i}^2)$ is replaced by a potentially smaller cost of just computing an $L^2$-projection.
Given that $N_{h_{i-1}} \approx 2^{-d} N_{h_{i}}$ (here $d=2$), a standard argument shows that 
$$T(i_{\max}) = O(N_{i_{\max}}^2)\ ,$$
that is, a cost that is proportional to that of a residual computation.

Another comment refers to a detail that is not very transparent in Algorithm 2, namely that coarse versions of 
$\lambda_{h_i}$ are necessary for each level, because 
$$\op{G}_{h_i}=I+(\op{D}^{h_i}_{1/\sqrt{\lambda_{h_i}}})^* (\op{K_{h_i}})^* \op{K_{h_i}}\op{D}^{h_i}_{1/\sqrt{\lambda_{h_i}}}\ .$$ 
Since the original problem is solved starting at the finest level, where $\lambda_{h_{i_{\max}}}$ is given by the
optimization algorithm, the functions $\lambda_{h_i}$ are obtained by simply discarding the values at finer nodes of $\lambda_{h_{i_{\max}}}$.
The parameter $\beta$ is hidden in $\lambda$ and affects the process indirectly.

%-----------SECTION-----------%
\section{Applications and numerical examples}
\label{mgimp:sec:num_examples}
%--------------------------%
In this section we discuss two applications. The first is related to the inverse contamination
problem studied in~\cite{Bart_Omar:SC05a, ABDGHV06}, where $\op{K}$ is a time-$T$
solution operator of a parabolic equation. The second  is a standard
elliptic-constrained optimal control problem with additional
box-constraints on the control.

\subsection{Solution strategies and metrics for success}
\label{ssec:metrics}
For both applications we apply Mehrotra's algorithm and we solve the inner linear systems
using the multigrid preconditioner previously defined.
In the absence of multiple grids, the linear systems~\eqref{lin_sys_normal_genWDH} are
solved using conjugate gradient (CG), while for more than one level we used
MG-preconditioned conjugate gradient squared (CGS), because of the slight
non-symmetry of the MG preconditioner.  As a first metric we record  the number of inner
linear iterations needed at each outer iteration; secondly, we record the
total number of finest-grid mat-vecs
for the entire solution process. Recall that each outer iteration
requires the solution of two linear systems with identical matrices,
namely one for the predictor step and one for the corrector step; in the interest of the presentation
we record only the linear iterations for the predictor step.  Also, a small number of 
mat-vecs are required in the process in addition to those needed for the predictor-corrector solves, and are reflected in
the count.
With respect to the second metric we remark that the proposed  algorithm is
intended for large-scale problems, with the most expensive computation
being the finest-scale residual computation. The
ultimate goal is to significantly reduce the total number  of
finest-scale mat-vecs, because this is expected to be directly linked with
execution time in a truly large-scale computation.  With regard to the second metric
we would like the total number of finest-level mat-vecs to decrease with $h\downarrow 0$.
As for the first metric (number of iterations) we would like to witness the following:\\
{\bf [a.]} The number of MG-preconditioned CGS-iterations should be
  less than half of the unpreconditioned CG-iterations  (each
  CGS-iteration involves two mat-vecs, while a CG iteration requires
  only one).\\
{\bf [b.]} For a given resolution $h$, the number of MG-preconditioned CGS
  iterations should be relatively bounded with respect to the number of
  levels used, provided the coarsest level is sufficiently fine, as
  stated in  Theorem~\ref{mgipm:th:multigrid}.\\
{\bf [c.]} Mostly important, {\bf the number of MG-preconditioned CGS iterations
  should decrease with $h\downarrow 0$}; in other words, the
  MG-preconditioned CGS becomes increasingly advantageous compared to CG as
  the problem-size increases. One word of caution though: linear systems of
  different resolutions are not necessarily related in a direct fashion,
  since their ``$\lambda$'' is dictated by the evolution of  Mehrotra's  algorithm,
  which is slightly different for  each resolution. For example, the tenth linear system
  to be solved in the IPM process at a resolution $2 h$ is not necessarily some
  coarse version of the tenth system to be solved at resolution $h$.\\
\indent Also, we should point out that in all our tests we use a cold start, that is, we do not
take advantage of results from coarser levels except for in the MG-solve of the inner linear
systems. While for realistic applications ``warm-start'' strategies are essential, we restrict our
attention to the way our multigrid technology plays a role in solving the inner linear systems.

%-----------SUBSECTION-----------%
\subsection{Time-reversal for a parabolic equation}
\label{ssec:invadr}
We consider the problem of finding the initial state (the control) for a system governed by a parabolic
equation given the state at a later time $T$ under additional box-constraints on the control.
\begin{comment}
An application of this problem is the question of identifying the initial concentration
of an air-contaminant subjected to an advection-diffusion equation that models
the evolution of the contaminant in an atmosphere with a known flow pattern, given
the entire state of the contaminant at a given time $T$. This is a simplified version of
the problem studied in~\cite{Bart_Omar:SC05a, ABDGHV06} where measurements of the contaminant were taken
at sparse locations in space but at all times, which is closer to a realistic scenario.
Since the problem is ill-posed, it is formulated as a regularized optimization
problem controlled by the initial state $v$. If the correct initial state is
localized, it is well known that lack of bounds on the control leads to a solution
that is not localized and can have values outside of the physically meaningful
$[0,1]$-range, thus we constrain the  control to satisfy $0\le v(x)\le 1$.
\end{comment}
Multigrid-preconditioning for the unconstrained version of this problem was
studied in detail in~\cite{mythesis, MR2429872} and for the space-time measurements in~\cite{Bart_Omar:SC05a}.

Formally, we consider the following parabolic initial value problem with periodic boundary
conditions
\begin{equation}
\label{eq:parabolic_gen}
\left\{
\begin{array}{lll}
\partial_t y  -\partial_x(a \partial_x y + b y)+ c y = 0 &,&\ \ \mbox{on}\; [0,1]\times (0, T]\ , \\
y(0,t)= y(1,t),\ \partial_x y(0,t)=\partial_x y(1,t)&,&\ \ \mbox{for}\; t\in (0,T]\ , \\
y(x, 0) =u(x)&,&\ \ \mbox{for}\; x\in [0,1]\ , \\
\end{array}
\right .
\end{equation}
where $a>0, b, c\ge 0$ are constants, and $T>0$ is the end-time. For $t>0$ we denote by
$\op{S}(t)\in \mathfrak{L}(L^2([0,1]))$ the time-$t$ solution operator mapping the initial value
onto $y(\cdot, t)$
$$u\stackrel{\op{S}(t)}{\longmapsto} y(\cdot, t)\ ,$$
and let $\op{K}=\op{S}(T)$. The discrete $\op{K}_h$ is obtained by using
a Galerkin formulation with continuous piecewise elements on a uniform grid for the spatial
discretization and Crank-Nicolson in time.
It is shown in~\cite{MR83m:65072} (see also~\cite{MR2249024}) that for sufficiently
small $h$ the following estimate holds:
\begin{equation}
\label{eq:parsmoothest}
\nnorm{\op{K} u - \op{K}_h\pi_h u}_{H^m}\le C h^{2-m} \nnorm{u}\ ,\ \forall u\in L^2([0,1])\ ,\ m=0, 1,
\end{equation}
where $C=C(T)$, provided the time step $k$ is proportional to the spatial resolution
\mbox{$k=C_1 h$}, with $C_1$ chosen to ensure stability. Consequently, space and time
resolutions are refined at the same rate, and Condition~\ref{mgipm:cond:condsmooth} is verified, so our theory applies.
The specific details (boundary conditions, constant advection etc.) in this example were chosen for convenience, however,
two and three spatial dimensions, other types of boundary conditions, as well as
smoothly varying functions in place of the constants $a, b, c$ are supported.

\paragraph{A direct verification of the convergence order} As mentioned earlier,
when running Mehrotra's algorithm with different resolutions, the added diagonal terms
$\lambda$ may not be in direct relationship with each other. Hence, in order to practically
verify the presence of $h^2$ in the estimate~\eqref{mgipm:eq:NGest} we resort to an artificial context:
we construct $\op{G}_h$ based on a fixed function $\lambda_h=\op{I}_h(\sin)+\beta$
for $h_j=80\cdot 2^j, j=0, 1, 2, 3$, and we define the corresponding
two-grid preconditioners $\op{N}_h$. Then we compute the ``distances''
$d_h=\max\{\abs{\ln \alpha} : \alpha\in \sigma(\op{G}_h, \op{N}_h)\}$. Since $d_h$
approximates the spectral distance of interest (because $\op{N}_h$ is close to being symmetric, actually
we have $d_h\le d_{\sigma}$) we expect to see that $d_h=O(h^2)$. We repeat the experiment for
$\beta=1, 0.1, 0.01$. The results presented in Table~\ref{tab:results_W3}, while not yet converged,
give a strong indication of an asymptotic rate $\lim_{h\rightarrow 0}d_{2h}/d_{h} = 4$.

\begin{table}[!ht]
\begin{center}
\caption{$d_h=\max\{\abs{\ln \alpha} : \alpha\in \sigma(\op{G}_h, \op{N}_h)\}$ for $\lambda(x)\approx (\sin(x) +\beta)$.}
\label{tab:results_W3}
\begin{tabular}{|l|l|l|l|l|l|l|}
  \hline
  $h\  \backslash\  \beta$& \multicolumn{2}{|c|}{1} & \multicolumn{2}{c|}{0.1}& \multicolumn{2}{c|}{0.01} \\\hline
        & $d_{h}$ & rate & $d_{h}$ & rate & $d_{h}$ & rate \\\hline
  1/80  & 0.0206 &         & 0.1127 &         &  0.2812 &        \\
  1/160 & 0.0066 &  3.1342 & 0.0363 &  3.1078 &  0.1270 & 2.2140 \\
  1/320 & 0.0020 &  3.3140 & 0.0102 &  3.5488 &  0.0445 & 2.8535 \\
  1/640 & 0.0006 &  3.5199 & 0.0027 &  3.7365 &  0.0123 & 3.6284 \\
  \hline
\end{tabular}
\end{center}
\end{table}

\paragraph{Numerical study}
We consider the ``true'' initial value $u_0$ supported on two intervals with $u_0$ reaching
the value $1$ on one of the intervals and $1/2$ on the other interval, then let $f=\op{K} u_0$.
Specific values are $a=4\cdot 10^{-3}, b=0.4, c=0$, and $T=0.8$. In Figure~\ref{fig:numerical_sol_invadr} we
show $u_0, f$ as well as the converged solution $u_{\min}$ of the box-constrained optimization problem with $\beta=10^{-3}$,
a value chosen because of the relatively good (visual) agreement of $u_0$ with $u_{\min}$. We run Mehrotra's algorithm
for $h=2^{-10}, 2^{-11}, 2^{-12}, 2^{-13}$, and for each $h$ we test the solvers with 1, 2, and 3 levels.
The number of linear iterations required by each of the linear solves in the predictor step are
shown in Figure~\ref{fig:lin_iterations_adr}, while the corresponding values for $\nnorm{\lambda^{-1/2}}_{W^2_{\infty}}$ and $\mu$ are shown in
the top two pictures of Figure~\ref{fig:muandlambdas_adr}.

\begin{figure}[!ht]
%\begin{center}
        \includegraphics[width=6.5in]{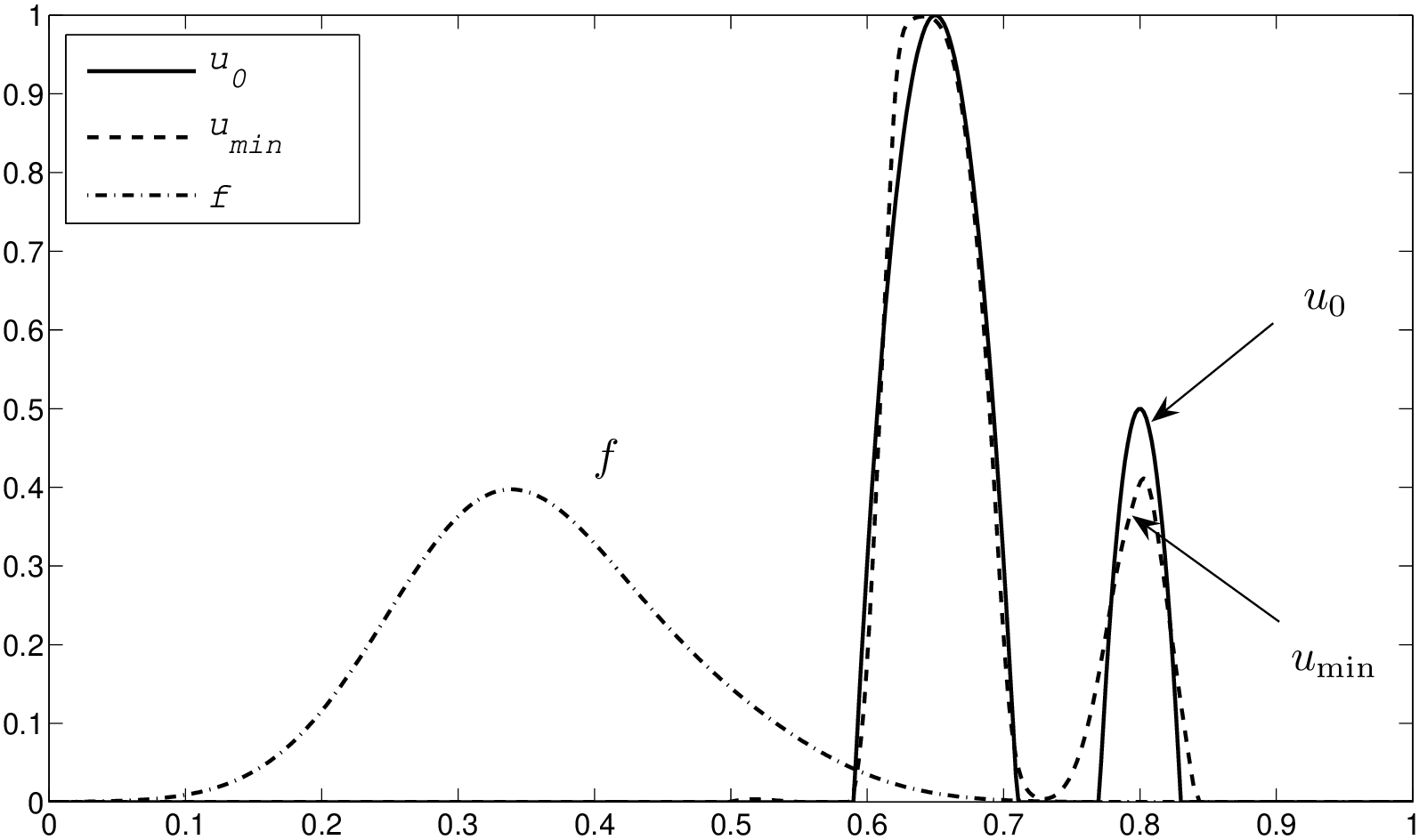}
\caption{Solution with $f=\op{K} u_0$, $\beta=10^{-3}$, $0\le u \le 1$.}
\label{fig:numerical_sol_invadr}
%\end{center}
\end{figure}

\begin{figure}[!th]
%\begin{center}
        \includegraphics[width=6.5in]{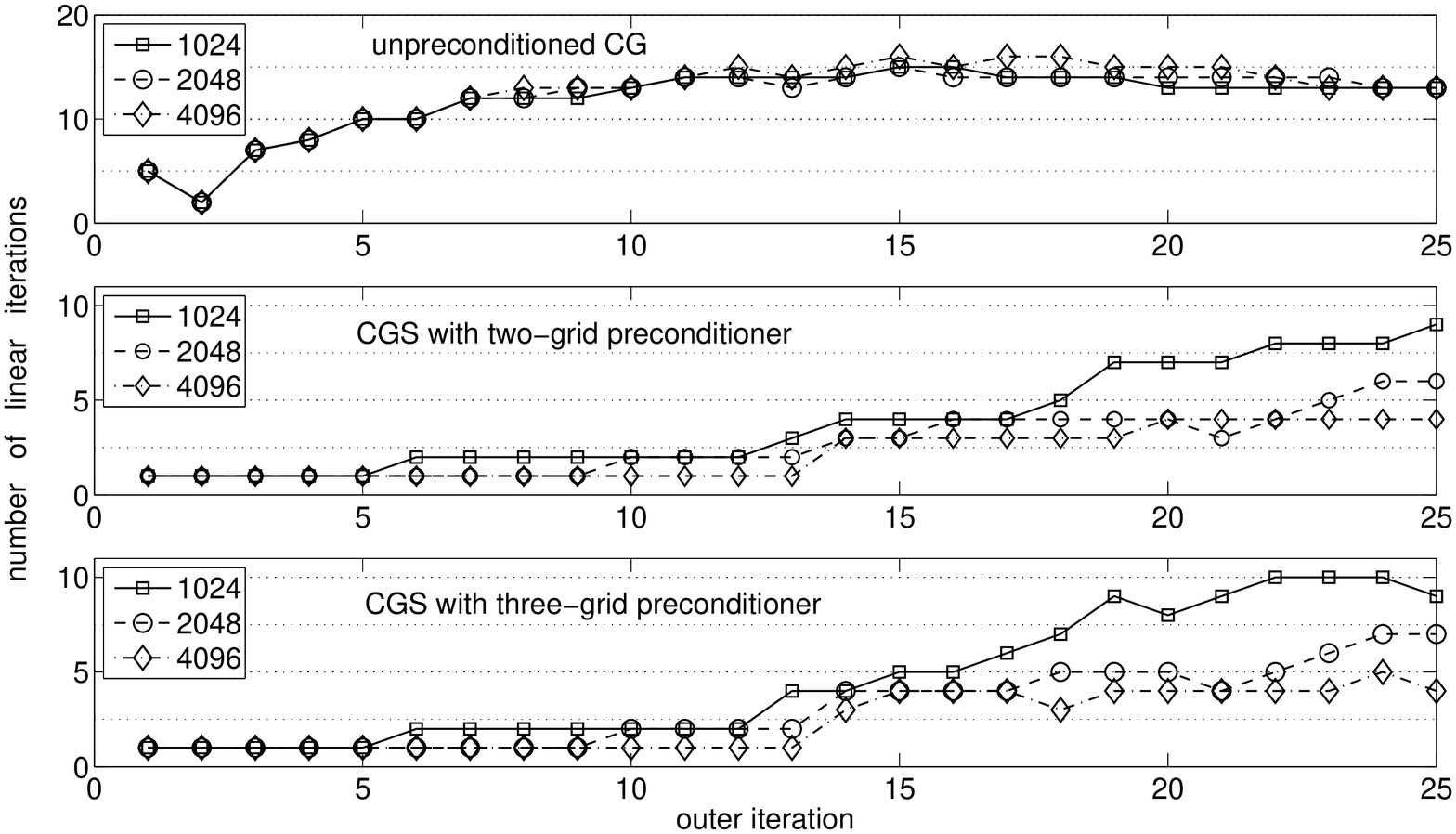}
\caption{Number of iterations for each of the predictor-step linear systems solved \textnormal{(}$\beta = 10^{-3}$\textnormal{)}.}
\label{fig:lin_iterations_adr}
%\end{center}
\end{figure}

\begin{figure}[!th]
%\begin{center}
        \includegraphics[width=6.5in]{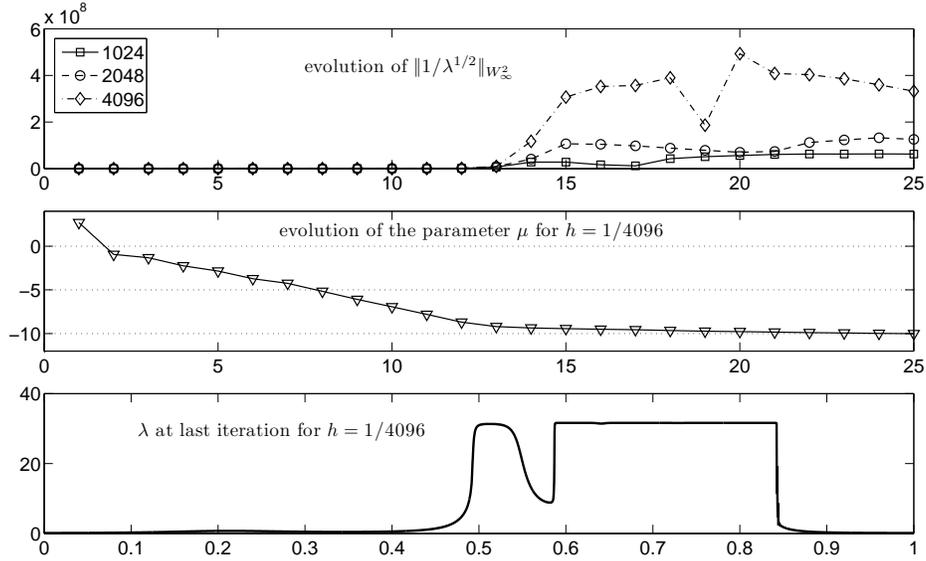}
\caption{Top: discrete $\nnorm{\lambda_h^{-\frac{1}{2}}}_{W^2_{\infty}}$ based on discrete Laplacian should give
  an idea of the size of $\nnorm{\lambda^{-\frac{1}{2}}}_{W^2_{\infty}}$; middle: $\log$-plot of $\mu$ as a function of
  outer iteration; bottom: $\lambda_h$ for $h=1/4096$ at the last outer iterate \textnormal{(}$\beta = 10^{-3}$\textnormal{)}.}
\label{fig:muandlambdas_adr}
%\end{center}
\end{figure}

First we remark that the number of unpreconditioned CG iterations appears to be mesh-independent (top chart in Figure~\ref{fig:lin_iterations_adr}):
essentially the curves representing the number of iterations for each of the resolutions more-or-less overlap.
%For example, the predictor-step linear system at the twentieth outer iteration requires between $13$ and $15$ iterations for all resolutions under  discussion.
We notice only a slight increase in number of iterations for higher resolutions.
Second, from the middle chart in  Figure~\ref{fig:lin_iterations_adr} representing the number of two-grid CGS iterations,
we infer that the number of two-grid preconditioned iterations {\bf consistently} decreases with $h\downarrow 0$, as desired. For example,
at the twenty-fourth iteration these numbers are
$8, 6$, and $4$, while at the twentieth they are $7, 4, 4$. This phenomenon is repeated for the three-grid preconditioner as can be seen from the bottom
chart in Figure~\ref{fig:lin_iterations_adr}. Moreover, for the most part, the number of MG-CGS preconditioned iterations is significantly smaller than
half the number of unpreconditioned CG iterations: e.g., for $h=1/4096$, up to the thirteenth iteration (where $\mu$ is already down to approx. $10^{-8}$, 
see Figure~\ref{fig:muandlambdas_adr})
 only one two-grid (or three-grid) MG-preconditioned CGS
iteration is necessary, while up to $15$ unpreconditioned CG iterations are needed. However, after the thirteenth iteration $\nnorm{\lambda^{-1/2}}_{W^2_{\infty}}$
shows a significant  increase, as seen on the top chart of Figure~\ref{fig:muandlambdas_adr},
and the MG-preconditioned CGS is less efficient: for $h=1/4096$, at the twenty-fourth  iteration $4$ two-grid iterations are needed
compared to $13$ CG iterations, a lesser advantage compared to the earlier outer iterations. The bottom chart in Figure~\ref{fig:muandlambdas_adr} shows
the last computed $\lambda_h$ at $h=1/4096$ to give an idea of why the quantity $\nnorm{\lambda^{-1/2}}_{W^2_{\infty}}$ is so large. A comparison between the
bottom and middle charts in Figure~\ref{fig:lin_iterations_adr} shows that the number of MG-preconditioned iterations is not very sensitive to the number
of levels, provided  the coarsest mesh is sufficiently fine. In this example four levels would force a much too coarse base mesh, and produce unsatisfactory
preconditioners. The last piece of evidence is the total count of finest-level mat-vecs, shown in Table~\ref{tab:finemat-vecs_invadr}. In this example,
a mat-vec involves solving the advection-reaction-diffusion equation on $[0,T]$.
The data clearly shows that,
as \mbox{$h\downarrow 0$}, the two-level solvers is getting increasingly efficient in  this metric compared to CG: the ratio goes from $581/728$ for $h=1/1024$
to $377/768$ for $h=1/8192=2^{-13}$. We should remark also that the essential impediment to a more significant improvement over CG lies in the increase
in the $\nnorm{\lambda^{-1/2}}_{W^2_{\infty}}$  as \mbox{$h\downarrow 0$}. As shown in Theorem~\ref{mgipm:th:twogrid}, non-smoothness of $\lambda^{-1/2}$ decreases the
preconditioner's efficiency.

\begin{table}[ht]
%\caption{Total number of finest-grid mat-vecs: left table -- reversed parabolic equation, right table -- elliptic constrained optimal control problem.}
%\label{tab:finemat-vecs_invadr}
\begin{minipage}[b]{0.5\linewidth}\centering
\caption{Total number of fine-grid mat-vecs\newline for the \textnormal{1D} reversed parabolic equation}
\label{tab:finemat-vecs_invadr}
\hspace{-1.cm}
\begin{tabular}{|l|l|l|l|}
  \hline
  $h\  \backslash\ $ levels& 1& 2& 3\\\hline
  1/1024  & 728 & 581 & 661 \\
  1/2048  & 740 & 463 & 489 \\
  1/4096  & 764 & 403 & 425 \\
  1/8192  & 768 & 377 & 403 \\
  \hline
\end{tabular}
\end{minipage}
\hspace{0.cm}
\begin{minipage}[b]{0.5\linewidth}
\centering
\caption{Total number  of fine-grid mat-vecs for \newline the \textnormal{2D} elliptic-constrained opt. ctrl. problem}
\label{tab:finemat-vecs_lap}
\begin{tabular}{|l|l|l|l|l|}
  \hline
  $h\  \backslash\ $ levels& 1& 2& 3& 4\\\hline
  1/256  & 354 & 282 & 572 & -- \\
  1/512  & 355 & 220 & 250 & 452 \\
  1/1024 & 355 & 198 & 210 & 224\\
%  1/1600 &  354 & 180 & 186  & 194\\
  1/2048 &  363 & 172 & 174  & 174\\
  \hline
\end{tabular}
\end{minipage}
\end{table}

%-----------SUBSECTION-----------%
\subsection{An elliptic-constrained control problem}
In this example we discuss the elliptic-constrained optimal control problem~\eqref{mgipm:eq:lap}
from Example B, which is a standard test problem in PDE-constrained
optimization~\cite{MR2505585, MR2516528}, and corresponds
to~\eqref{mgipm:eq:maineq} with $\op{K} = \Delta^{-1}$. We consider a square domain
with a continuous piecewise linear finite element discretization based on the standard three-lines
triangular mesh\footnote{The {\it three-line mesh} is obtained by dividing the square into
equally sized squares  with sides parallel to the coordinate axes, and by further cutting each little
square along its slope-one diagonal.}. Standard estimates for finite element solutions of elliptic problems
show that Condition~\ref{mgipm:cond:condsmooth} is verified~\cite{MR2373954}.

\paragraph{Numerical study}
Let $\Omega=[0,1]\times [0,1]$, $\beta=10^{-6}$, and
$f$ be the function that satisfies  $\Delta f = u_0, f|_{\partial
\Omega}=0$,  where $u_0(x,y)= \frac{3}{2} \sin(2 \pi x)\: \sin(2 \pi
y)$. With this selection of $f$,  the choice $u=u_0$ would be a
solution of~\eqref{mgipm:eq:lap} if $\beta=0$ and no box constraints
were present (or if $[-\frac{3}{2}, \frac{3}{2}]\subseteq [\underline{u},\overline{u}]$).   Here the bounds
$[\underline{u},\overline{u}]=[-1, 1]$ are active: without them, given that $\beta\ll 1$, the
solution would be close to $u_0$, that is,  would have a maximum
(resp. minimum) close to $3/2$ (resp. $-3/2$).  The solution with
$h=1/128$ is depicted in Figure~\ref{fig:numerical_sol}.  We have
solved the problem with $h=2^{-8}, 2^{-9}, 2^{-10}, 2^{-11}$ using one, two,
three, and four levels (where appropriate) using the strategy
described in  Section~\ref{ssec:metrics}.

\begin{figure}[!ht]
%\begin{center}
        \includegraphics[width=5.5in]{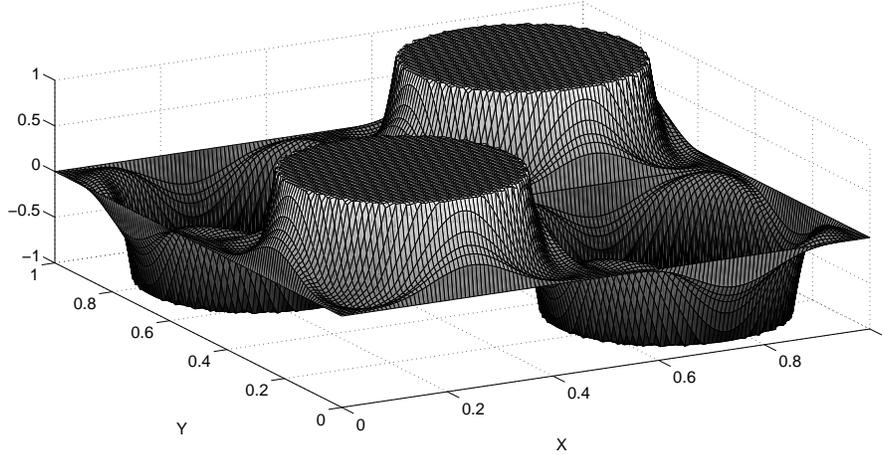}
\caption{Solution with $f$ satisfying $\Delta f=\frac{3}{2} \sin(2 \pi x)\: \sin(2 \pi y)$, $\beta=10^{-6}$, $[\underline{u},\overline{u}]=[-1,1]$.}
\label{fig:numerical_sol}
%\end{center}
\end{figure}

\begin{figure}[!th]
%\begin{center}
        \includegraphics[width=6.5in]{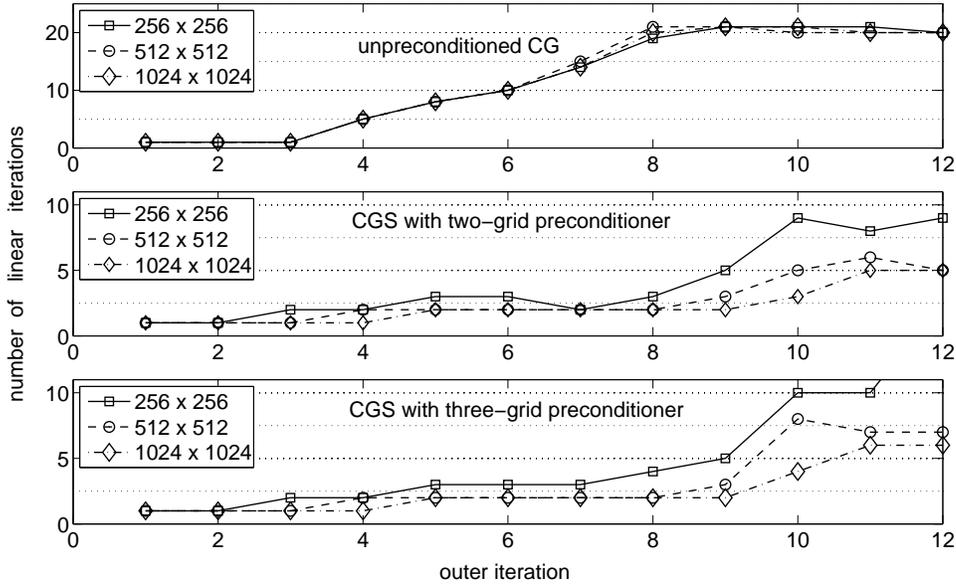}
\caption{Number of iterations for each of the predictor-step linear systems solved \textnormal{(}$\beta = 10^{-6}$\textnormal{)}.}
\label{fig:lin_iterations}
%\end{center}
\end{figure}

As with the previous example, we show in
Figure~\ref{fig:lin_iterations} the number of iterations required by
each of the linear systems at the predictor step. The top plot shows
the number of unpreconditioned iterations to level off at 21. The
middle  plot again shows two facts:  the number of
MG-preconditioned CGS iterations decreases with $h\downarrow 0$. In
addition, for the finest grid,  the number of two-grid preconditioned
CGS iterations is less than $1/4$ that of the number of iterations for the unpreconditioned
case even when looking beyond  the eighth outer iteration, where as
before, roughness of the $\lambda$-function lowers the quality of the
MG-algorithm. For example, for the tenth outer iteration with
$h=1/1024$, CGS required $3$ iterations, while $21$ iterations
were needed in the unpreconditioned case; for the  eleventh outer iteration the numbers are $5$
vs.~$20$. Of course, each two-grid preconditioned CGS iteration is significantly more
expensive than an unpreconditioned CG iteration, which is why the number of levels
should be maximized. The bottom plot in
Figure~\ref{fig:lin_iterations} shows that the two-level  behavior is
replicated using three-level preconditioners. Moreover, {\bf at fine
resolutions} (\mbox{$h\le 2^{-10}$}), the numbers  of required three-level preconditioned CGS iterations are
not significantly higher than those of two-level CGS
iterations. However, this is not the case for low resolution
$h=2^{-8}$, where the number of three-level CGS iterations required for
the last three systems (not shown on the plot) is quite large: $45,
17, 58$. This is why we insist that the  MG-preconditioner is
efficient only when the coarsest resolution used is sufficiently fine
and the finest resolution $h$ is  small. With respect to the second
metric, we show in Table~\ref{tab:finemat-vecs_lap} the total number of
fine mat-vecs for each of the runs, and the results confirm that
the MG-preconditioner becomes increasingly efficient with $h\downarrow 0$.
Recall that for this application a mat-vec requires solving the Poisson equation.

\appendix
\section{Some facts about the spectral distance}
\label{sec:spec_dist}
Throughout this section $(\op{X}, \innprd{\cdot}{\cdot})$ is a real, finite
dimensional Hilbert space with norm $\nnorm{\cdot}$.
All operators in this section are assumed to be in $\op{L}_+(\op{X})$ (see Section~\ref{ssec:tganalysis} for definition)
unless otherwise specified.
\noindent The following inequalities were proved in~\cite{MR2429872}  (Lemma 3.2):
% LEMMA - log ineq
\begin{lemma}
\label{lma:log_ineq}
If $\alpha\in(0,1)$ and $z\in {\mathcal B}_{\alpha}(1)$, then
\begin{equation}
\label{eq:log_ineq}
\frac{\ln(1+\alpha)}{\alpha}|1-z| \le |\ln z| \le
\frac{|\ln(1-\alpha)|}{\alpha} |1-z|\ .
\end{equation}
For $\abs{\ln z}\le \delta$ we have
\begin{equation}
\label{eq:log_ineq2}
\frac{1-e^{-\delta}}{\delta} |\ln z|\le |1-z| \le \frac{e^{\delta}-1}{\delta} |\ln z| .
\end{equation}
\end{lemma}
%
% LEMMA
%
\begin{lemma}
\label{lma:dist_ctrl_radius}
Let $L, G\in \op{L}_+(\op{X})$ such that 
$$\min\left(d_{\sigma}(L^{-1}, G), d_{\sigma}(L, G^{-1})\right)\le \delta\ .$$ 
Then
\begin{equation}
\label{eq:spec_rad_vs_spec_dist}
\rho(I-L G)\le \frac{e^{\delta}-1}{\delta} \min\left(d_{\sigma}(L^{-1}, G), d_{\sigma}(L, G^{-1})\right)\ .
\end{equation}
\end{lemma}
%
% BEGIN PROOF
\begin{proof}
If $\lambda\in\sigma(I-L G)$ then there exists a unit vector $u\in \op{X}^{\mathbb{C}}$
such that $(I-L G)u = \lambda u$, therefore
\begin{eqnarray}
\label{eq:srd0}
(1-\lambda)u = L G u\ .
\end{eqnarray}
After left-multiplying  with $L^{-1}$ and taking the inner product with $u$ we obtain
$$(1-\lambda)\innprd{L^{-1}u}{u} = \innprd{G u}{u},\ \ \mathrm{therefore}\ \ 
\lambda = 1-\frac{\innprd{G u}{u}}{\innprd{L^{-1} u}{u}}\ .$$
If we substitute $v=G^{-1}u$ in~\eqref{eq:srd0} and take the inner product with $v$ we have
$$
(1-\lambda)G^{-1}v = L v\ ,\ \ \mathrm{therefore}\ \lambda = 1-\frac{\innprd{L v}{v}}{\innprd{G^{-1} v}{v}}\ .
$$
Hence, if $d_{\sigma}(L^{-1}, G)\le \delta$, then
\begin{eqnarray*}
\rho(I-L G)&\le& \sup\{|1-z|\ :\ z = \innprd{G u}{u}/\innprd{L^{-1} u}{u}
\ \mathrm{for\ some\ }u\in \op{X}^{\mathbb{C}}\setminus\{0\}\}\\
&\stackrel{\eqref{eq:log_ineq2}}{\le}& 
\frac{e^{\delta}-1}{\delta} d_{\sigma}(L^{-1}, G)\ .
\end{eqnarray*}
Instead, if $d_{\sigma}(L, G^{-1})\le \delta$, then
\begin{eqnarray*}
\rho(I-L G)&\le& \sup\{|1-z|\ :\ z = \innprd{L u}{u}/\innprd{G^{-1} u}{u}
\ \mathrm{for\ some\ }u\in \op{X}^{\mathbb{C}}\setminus\{0\}\}\\
&\stackrel{\eqref{eq:log_ineq2}}{\le}& 
\frac{e^{\delta}-1}{\delta} d_{\sigma}(L, G^{-1})\ .
\end{eqnarray*}
which proves~\eqref{eq:spec_rad_vs_spec_dist}.
\end{proof}\\
%
% BEGIN PROOF
%
% THE REST IS OUT
\begin{comment}
We now give the proof of  Lemma~\ref{lma:norm_dist}:
\begin{proof}
(a) For $u\in\op{X}^{\mathbb{C}},\ \nnorm{u}=1$ we have
\begin{eqnarray*}
\Abs{\innprd{G u}{u}}\ge \op{Real}{\innprd{G u}{u}} = \innprd{G_s u}{u} \ge \min_{\nnorm{u}=1} \innprd{G_s u}{u} = \nnorm{G_s^{-1}}^{-1}\ ,
\end{eqnarray*}
so 
\begin{equation}
\label{eq:ineq1}
\Abs{\frac{\innprd{N u}{u}}{\innprd{G u}{u}} - 1} \le \frac{{\nnorm{N-G}}}{\Abs{\innprd{G u}{u}}}\le
\nnorm{G_s^{-1}}\cdot {\nnorm{N-G}}\ .
\end{equation}
Hence  $\Abs{\frac{\innprd{N u}{u}}{\innprd{G u}{u}} - 1}\le\alpha$, and
\begin{eqnarray*}
\Abs{\ln\frac{\innprd{N u}{u}}{\innprd{G u}{u}}} \stackrel{\eqref{eq:log_ineq}}{\le}\frac{|\ln(1-\alpha)|}{\alpha}\Abs{\frac{\innprd{N u}{u}}{\innprd{G u}{u}} - 1}
\stackrel{\eqref{eq:ineq1}}{\le}\frac{|\ln(1-\alpha)|}{\alpha}\nnorm{G_s^{-1}}^{-1}\cdot {\nnorm{N-G}}\ .
\end{eqnarray*}
(b) Similar arguments lead to 
\begin{eqnarray*}
d_{\sigma}(L, G^{-1})  &=& d_{\sigma}(G^{\frac{1}{2}} L G^{\frac{1}{2}}, I) \le\frac{\abs{\ln (1-\alpha)}}{\alpha} \nnorm{G^{\frac{1}{2}} L G^{\frac{1}{2}} - I}\\
&\le&\frac{\abs{\ln (1-\alpha)}}{\alpha} \nnorm{G^{-\frac{1}{2}}} \cdot \nnorm{G L - I}\cdot \nnorm{ G^{\frac{1}{2}}}\ ,
\end{eqnarray*}
and the conclusion follows.
\end{proof} \\
\end{comment}
%We should point out that the above estimates may not be very accurate in certain situations. For further
%results on the spectral distance please consult~\cite{MR2429872}.

\bibliography{ref,ref_ipm}
\bibliographystyle{siam}

\end{document}